\documentclass[11pt]{article}

\usepackage{amsfonts}
\usepackage{amscd}
\usepackage{amssymb}
\usepackage{amsthm}
\usepackage{amsmath}
\usepackage{stmaryrd}
\usepackage{graphicx}
\usepackage{color}
\usepackage{verbatim}

\input prepictex
\input pictex
\input postpictex

 \theoremstyle{plain}
\newtheorem{thm}{Theorem}[section]
\newtheorem{lemma}[thm]{Lemma}
\newtheorem{prop}[thm]{Proposition}
\newtheorem{cor}[thm]{Corollary}
\newtheorem*{thmB}{Theorem B}
\newtheorem*{thmS}{Theorem S}
\newtheorem*{thmR}{Theorem R}
\newtheorem*{thmG}{Theorem G}
\newtheorem*{thmC}{Theorem C}
\newtheorem*{thmCd}{Theorem $\mathbf{C}'$}

\theoremstyle{definition}
\newtheorem{Example}[thm]{Example}
\newtheorem*{defn}{Definition}
\newtheorem{remark}[thm]{Remark}
\theoremstyle{remark}

\numberwithin{equation}{section}

\def\cA{\mathcal{A}}

\def\cG{\mathcal{G}}
\def\cH{\mathcal{H}}

\def \cP{\mathcal{P}}

\def\CC{\mathbb{C}}

\def\FF{\mathbb{F}}

\def\RR{\mathbb{R}}

\def\ZZ{\mathbb{Z}}

\def\fg{\mathfrak{g}}
\def\fh{\mathfrak{h}}

\def\Card{\mathrm{Card}}

\def\dim{\mathrm{dim}}

\def\la{\lambda}

\def\nla{\beginpicture        
        \setplotarea x from -1 to 1, y from -1 to 1  
        \arrow <5pt> [.2,.67] from 0.25 0.12 to -0.25 -0.12
    \endpicture}

\def\upra{\beginpicture        
        \setplotarea x from -1 to 1, y from -1 to 1  
        \arrow <5pt> [.2,.67] from -0.25 -0.12 to 0.25 0.12
    \endpicture}

\def\ra{\beginpicture         
        \setplotarea x from -1 to 1, y from -1 to 1  
        \arrow <5pt> [.2,.67] from -0.25 0.12 to 0.25 -0.12
    \endpicture}

\def\upla{\beginpicture         
        \setplotarea x from -1 to 1, y from -1 to 1  
        \arrow <5pt> [.2,.67] from 0.25 -0.12 to -0.25 0.12
    \endpicture}

\def\da{\beginpicture    
        \setplotarea x from -1 to 1, y from -1 to 1  
        \arrow <5pt> [.2,.67] from 0 0.3 to 0 -0.3
    \endpicture}

\def\ua{\beginpicture    
        \setplotarea x from -1 to 1, y from -1 to 1  
        \arrow <5pt> [.2,.67] from 0 -0.3 to 0 0.3
    \endpicture}

\def\lfold{\beginpicture 
\setplotarea x from -1 to 1, y from -1 to 1  
    \plot 0 0 -0.3 -0.15 /
    \plot -0.3 -0.15 -0.225 -0.3 /
    \arrow <5pt> [.2,.67] from -0.225 -0.3 to 0.05 -0.13
\endpicture}

\def\lfoldalt{\beginpicture 
\setplotarea x from -1 to 1, y from -1 to 1  
    \plot 0.05 -0.13 -0.225 -0.3 /
    \plot -0.225 -0.3 -0.3 -0.15 /
    \arrow <5pt> [.2,.67] from -0.3 -0.15 to 0 0
\endpicture}

\def\rfold{\beginpicture
\setplotarea x from -1 to 1, y from -1 to 1  
    \arrow <5pt> [.2,.67] from 0.3 -0.15 to 0 0
    \plot 0.3 -0.15 0.225 -0.3 /
    \plot 0.225 -0.3 -0.05 -0.13 /
\endpicture}

\def\dfold{\beginpicture    
\setplotarea x from -1 to 1, y from -1 to 1  
    \plot 0.08 0 0.08 -0.32 /
    \plot 0.08 -0.32 -0.08 -0.32 /
    \arrow <5pt> [.2,.67] from -0.08 -0.32 to -0.08 0
\endpicture}

\def\rfoldalt{\beginpicture
\setplotarea x from -1 to 1, y from -1 to 1  
    \plot 0 0 0.3 -0.15 /
    \plot 0.3 -0.15 0.225 -0.3  /
    \arrow <5pt> [.2,.67] from 0.225 -0.3 to -0.05 -0.13
\endpicture}

\def\dfolddash{\beginpicture    
\setplotarea x from -1 to 1, y from -1 to 1  
    \arrow <5pt> [.2,.67] from 0.08 -0.32 to 0.08 0 %
    \plot 0.08 -0.32 -0.08 -0.32 /
    \plot -0.08 -0.32 -0.08 0 /
\endpicture}

\definecolor{Gray}{gray}{0.5}

\makeatletter
\renewcommand{\@makefnmark}{\mbox{\textsuperscript{}}}
\makeatother

\title{Alcove walks, buildings, symmetric functions and representations}
\author{
James Parkinson \\
Institut f\"{u}r Mathematische Strukturtheorie \\ 
Technische Universit\"{a}t Graz
\and
Arun Ram \\
Department of Mathematics and Statistics \\
University of Melbourne \\
Parkville VIC 3010 Australia \\
A.Ram@ms.unimelb.edu.au \\
and \\
Department of Mathematics\\ University of Wisconsin\\
Madison, WI 53706 USA \\ 
ram@math.wisc.edu 
}

\begin{document}

\maketitle


\begin{abstract} 
For a complex simple Lie algebra $\fg$, the dimension $K_{\lambda\mu}$ of the $\mu$ weight space of a finite dimensional representation of highest weight $\lambda$ is the same as the number of Littelmann paths of type $\lambda$ and weight $\mu$. In this paper we give an explicit construction of a path of type $\lambda$ and weight $\mu$ whenever $K_{\lambda\mu}\ne 0$. This construction has additional consequences, it produces an explicit point in the building which chamber retracts to $\lambda$ and sector retracts to~$\mu$, and an explicit point of the affine Grassmannian in the corresponding Mirkovi\'c-Vilonen intersection. In an appendix we discuss the connection between retractions in buildings and alcove walks.
\end{abstract}

\section*{Introduction}

Let $R$ be a reduced irreducible root system with Weyl group $W_0$, coroot lattice~$Q$, and coweight lattice~$P$. Let $W=Q\rtimes W_0$ be the affine Weyl group. For $\la\in P^+$ a dominant coweight, let $\Pi_{\la}$ be the saturated set with highest coweight~$\la$, so that 
\begin{align*}
\Pi_{\la}=\bigcap_{w\in W_0}w(\la-Q^+),
\end{align*}
where $Q^+=\ZZ_{\geq0}\alpha_1^{\vee}+\cdots+\ZZ_{\geq0}\alpha_n^{\vee}$ and $\alpha_1^{\vee},\ldots,\alpha_n^{\vee}$ are the simple coroots of~$R$. 

Consider the following well known theorems from \textbf{R}epresentation theory, \textbf{S}ymmetric function theory, \textbf{B}uilding theory, and \textbf{G}eometry/\textbf{G}roup theory.

\begin{thmR} Let $\fg$ be the finite dimensional semisimple Lie algebra over $\CC$ with root system~$R^{\vee}$, and for $\la\in P^+$ let $V(\la)$ be the finite dimensional irreducible representation of $\fg$ with highest weight $\la$. For $\mu\in P$, let $V(\la)_{\mu}$ be the $\mu$-weight space of $V(\la)$. Then  
$\dim V(\la)_{\mu}\neq 0$ if and only if $\mu\in\Pi_{\la}$.
\end{thmR}

\begin{thmS} For $\lambda\in P^+$ let $c_{\lambda\mu}(q^{-1})$, $\mu\in P$, be the 
coefficients of the Macdonald spherical function
$P_\lambda(x,q^{-1})$ (see (\ref{eq:Pladefn})) so that
$$
P_{\la}(x, q^{-1})=\sum_{\mu\in P}c_{\la\mu}(q^{-1})x^{\mu},
\qquad \hbox{with $c_{\lambda\mu}(q^{-1})\in \ZZ[q^{-1}]$.}
$$ 
Then $c_{\la\mu}(q^{-1})\neq 0$ if and only if $\mu\in\Pi_{\la}$.
\end{thmS}

\begin{thmB} Let $X$ be a thick affine building of type $W$. Let $S_0$ be a sector in an apartment $A_0$. Identify $A_0$ with $W$ such that $S_0$ is identified with~$w_0C_0$, 
where $w_0$ is the longest element of $W_0$ and $C_0$ is the Weyl chamber of $W$ which 
contains $1\in W$. 
Let $\rho_1$ be the retraction of $X$ onto $A_0$ centred at the chamber~$1$ and let 
$\rho_{-\infty}$ be the retraction onto $A_0$ centred at the sector~$S_0$. Then for $\la\in P^+$,
$$
\{x\in X\mid \rho_1(x)\in W_0\la \textrm{ and } \rho_{-\infty}(x)=\mu\}
\neq\emptyset
\qquad\textrm{if and only if}\qquad \mu\in \Pi_{\la}.
$$
\end{thmB}

\begin{thmG} Let $G(\FF)$ denote the Chevalley group constructed from the 
data~$(R,P,\FF)$, where $\FF$ is a ring or field (cf. \cite{steinberg}). 
Let $k$ be a field and let $k((t))$ be the field of Laurent power series. 
Let $G=G(k((t)))$ be the loop group, $U^-$ the subgroup of~$G$ generated by the negative root subgroups $U_{-\alpha}=\{x_{-\alpha}(f)\mid f\in k((t))\}$ with $\alpha\in R^+$, and let $K=G(k[[t]])$, where $k[[t]]$ is the ring of Taylor series. Then for $\la\in P^+$,
$$
U^-t_{\mu}K\cap Kt_{\la}K\neq\emptyset\qquad\textrm{if and only if}\qquad \mu\in\Pi_{\la},
$$
where for $\nu\in P$, $t_{\nu}\in G$ is a ``diagonal matrix'' representing the translation $t_{\nu}\in P\rtimes W_0$.
\end{thmG}

It is known that Theorem~$\mathbf{R}$ implies Theorems~$\mathbf{S}$, $\mathbf{B}$ and 
$\mathbf{G}$, and also that Theorems~$\mathbf{S}$, $\mathbf{B}$ and $\mathbf{G}$ are equivalent 
to each other (see Proposition~\ref{prop:equiv} below). At first glance this might seem surprising. 
For example, to see that Theorem~$\mathbf{B}$ implies Theorem~$\mathbf{G}$ take $X$ to be the 
affine building of $G(k((t)))$ and interpret the intersection of double cosets in Theorem~$\mathbf{G}$ 
in terms of retractions in $X$ (as in \cite[Corollary~4.1]{petra}). However,  the reverse implication is not obvious for not all affine 
buildings can be constructed from group theory. The key to understanding the relationships between Theorems~$\mathbf{R}$, $\mathbf{S}$, $\mathbf{B}$ and $\mathbf{G}$ comes from the combinatorial theory of~\textit{path models}/\textit{alcove walks}. The statement which is equivalent to Theorems~$\mathbf{S}$, $\mathbf{B}$ and $\mathbf{G}$ in the alcove walk language is:

\begin{thmC} Let $\la\in P^+$ and fix a minimal length alcove walk $\vec m_{\la}$ to the minimal length element $m_{\la}\in W_0t_{\la}W_0$. Let $W^{\la}_0$ be a set of minimal length representatives for cosets in $W_0/W_{0\la}$, where $W_{0\la}=\{w\in W_0\mid w\la=\la\}$, and fix a minimal length walk $\vec u$ to each $u\in W_0^{\la}$. For $\mu\in P$ let
$$
\cP(\vec\la)_{\mu}=\left\{\begin{matrix}\textrm{positively folded alcove walks of type}\\
\textrm{$\vec u\cdot\vec m_{\la}$ with end alcove in $t_{\mu}W_0$}\end{matrix} \,\,\bigg|\,\, u\in W_0^{\la}\right\}.
$$
Then $\cP(\vec\la)_{\mu}\neq\emptyset$ if and only if $\mu\in\Pi_{\la}$.
\end{thmC}

In this paper we use ideas from the theory of crystal bases (see \cite{littelmanncrystals}) to give a simple combinatorial proof of Theorem~$\mathbf{C}$. In fact we prove the following stronger theorem, which is the combinatorial equivalent to Theorem~$\mathbf{R}$.

\begin{thmCd} If $\mu\not\in \Pi_{\la}$ then $\cP(\vec\la)_{\mu}=\emptyset$ and if $\mu\in\Pi_{\la}$ then there exists a path $p\in\cP(\vec\la)_{\mu}$ with ``optimal dimension'' $\langle\la+\mu,\rho\rangle$, where $\rho=\frac{1}{2}\sum_{\alpha\in R^+}\alpha$.
\end{thmCd}

A motivation for providing this proof is that, as illustrated in the proof of Proposition~\ref{prop:equiv}, we see the theory of path combinatorics as an integral thread holding together the building theory, symmetric function theory, geometry, and representation theory. Thus it is desirable to see the proofs of Theorems~$\mathbf{S}$, $\mathbf{B}$, $\mathbf{G}$, and~$\mathbf{R}$ come out of this central theory. Furthermore, we are interested here in giving a proof of Theorem~$\mathbf{C}$ that can be translated into abstract building theory. For this purpose we discuss the connection between labelled alcove walks and retractions in abstract buildings in Appendix~\ref{app:buildings}.

Currently most of the existing proofs of Theorems~$\mathbf{S}$, $\mathbf{B}$ or~$\mathbf{G}$ appeal to Theorem~$\mathbf{R}$ (which appears in many introductory books on representation theory, for example \cite[Proposition~21.3]{h2}). For example, the proof of Theorem~$\mathbf{B}$ in \cite{petra} uses Theorem~$\mathbf{R}$ and the character formulae from~\cite{litt}, and during the proof of \cite[Theorem~3.2]{MV3}, Theorem~$\mathbf{G}$ is applied with Theorem~$\mathbf{R}$ given as the reference. The proofs of Theorems~$\mathbf{S}$, $\mathbf{B}$ or~$\mathbf{G}$ that are of a combinatorial flavour are generally quite complex. For example, the proofs of Theorem~$\mathbf{S}$ in the series of papers \cite{haines}, \cite{rapoport} and \cite{tupan} use fairly specific details about groups related to those in Theorem~$\mathbf{G}$.

Our constructive proof of Theorem~$\mathbf{C}'$ gives a lower bound for the cardinality of the set in Theorem~$\mathbf{B}$, and a dense subset of a Mirkovi\'{c}-Vilonen cycle with ``explicit'' string parameters in Theorem~$\mathbf{G}$, and, in the case of Theorem~$\mathbf{S}$,  a lower bound 
for $c_{\lambda\mu}(q^{-1})$ when $q\in\RR_{>1}$ (see Corollary~\ref{cor:implications}).

J. Parkinson is supported under the Austrian Science Fund grant FWF-P18703-N18, and thanks the Technische Universit\"{a}t Graz for its hospitality in supporting his research. This research was also partially supported by the National
Science Foundation (NSF) under grant DMS-0353038 at the University of Wisconsin, Madison,
and this paper was completed while A. Ram was in residence at  the
special semester in Combinatorial Representation Theory at Mathematical Sciences
Research Institute (MSRI).  It is a pleasure to thank MSRI for hospitality, support and a wonderful
and stimulating working environment. Finally, J. Parkinson thanks the organisers of the workshop \textit{Buildings: Interactions with Algebra and Geometry}, Oberwolfach, January 2008, which motivated the publication of this paper.


\section{Equivalence of the theorems}

Before continuing with the main part of the paper let us briefly survey how Theorems~$\mathbf{R}$, $\mathbf{S}$, $\mathbf{B}$ and~$\mathbf{G}$ are related to Theorems~$\mathbf{C}$ and~$\mathbf{C}'$. 
The language of positively folded alcove walks is reviewed in Section~\ref{section:alcovewalks}. 

\begin{prop}\label{prop:equiv} Theorems~$\mathbf{S}$, $\mathbf{B}$ and $\mathbf{G}$ are all equivalent to Theorem~$\mathbf{C}$, and Theorem~$\mathbf{R}$ is equivalent to Theorem~$\mathbf{C}'$.
\end{prop}

\begin{proof} \textbf{Theorem~$\mathbf{S}$ $\Longleftrightarrow$ Theorem $\mathbf{C}$.} We will assume for simplicity that we are in the case of one parameter~$q$, although this assumption is easily removed. The Macdonald spherical function $P_{\la}(x,q^{-1})$ is defined by
\begin{align}\label{eq:Pladefn}
P_{\la}(x,q^{-1})=\frac{1}{W_{0\la}(q^{-1})}\sum_{w\in W_0}w\bigg(x^{\la}\prod_{\alpha\in R^+}\frac{1-q^{-1}x^{-\alpha^{\vee}}}{1-x^{-\alpha^{\vee}}}\bigg),
\end{align}
where $W_{0\la}=\{w\in W_0\mid w\la=\la\}$ and $W_{0\la}(q^{-1})=\sum_{w\in W_{0\la}}q^{-\ell(w)}$. The equivalence between Theorem~$\mathbf{S}$ and Theorem~$\mathbf{C}$
follows from the following formula of Schwer \cite{schwer} (also see \cite{ram2})
\begin{align}\label{eq:macpath}
P_{\la}(x,q^{-1})=\sum_{\mu\in P}\bigg(\sum_{p\in\cP(\vec\la)_{\mu}}q^{-(\langle\la+\mu,\rho\rangle-\dim(p))}(1-q^{-1})^{f(p)}\bigg)x^{\mu},
\end{align} 
where $f(p)$ is the number of folds in $p$, and 
$\dim(p)$ is the number of positive crossings in $p$ plus the number of folds in $p$.  
As observed in \cite[(4.5)]{schwer}, the formula (\ref{eq:macpath}) shows that 
$\dim(p)\leq\langle\la+\mu,\rho\rangle$ for all $p\in\cP(\vec\la)_{\mu}$. 
We will give a combinatorial proof of this fact in Lemma~\ref{lem:bound} 
(see also \cite[\S5 Proposition~4]{litt}). We sketch a proof of~(\ref{eq:macpath}) in
Appendix~\ref{app:pathformula}.

\smallskip

\noindent\textbf{Theorem~$\mathbf{B}$ $\Longleftrightarrow$ Theorem $\mathbf{C}$.} 
Let $\cG(\vec\la)$ be the set of all galleries in~$X$ of types $\vec u\cdot\vec m_{\la}$ with 
$u\in W_0^{\la}$ starting at the chamber $1$ (there is a slight complication here 
because $m_{\la}\in P\rtimes W_0$ is not necessarily in $W$; 
see the proof of Theorem~\ref{thm:retractions}). 
Mapping galleries to their end vertices gives a bijection between $\cG(\vec\la)$ and 
$\{x\in X\mid\rho_1(x)\in W_0\la\}$. By \cite[Proposition~3.3]{petra} (see also Appendix~\ref{app:buildings}), 
$\rho_{-\infty}(\cG(\vec\la))=\cP(\vec\la)$, where $\cP(\vec\la)$ is the set of all positively folded 
alcove walks of types $\vec u\cdot\vec m_{\la}$ with $u\in W_0^{\la}$. The thickness of the building 
is crucial here to guarantee the existence of the third chamber hanging off a panel that 
``folds back'' under the retraction~$\rho_{-\infty}$. We make the arguments here more 
precise in Appendix~\ref{app:buildings}.

\smallskip

\noindent\textbf{Theorem~$\mathbf{G}$ $\Longleftrightarrow$ Theorem $\mathbf{C}$.} Let $B(k)$ be the standard Borel subgroup of $G(k)$, and let $I=\mathrm{ev}_{t=0}^{-1}(B(k))$ be the standard Iwahori subgroup of $G(k((t)))$, where $\mathrm{ev}_{t=0}:k[[t]]\to k$ is evaluation at~$t=0$. Theorem~7.1 in \cite{ramparkinsonschwer} gives a bijection between the $G/I$-points of $U^-vI\cap IwI$ (with $v,w\in P\rtimes W_0$) and ``labelled'' positively folded alcove walks of type $\vec w$ with end alcove~$v$. Using the decomposition $K=\bigsqcup_{w\in W_0}IwI$ one has
$$
U^-t_{\mu}K\cap Kt_{\la}K=\bigsqcup\, U^-vI\cap IwI,
$$
where the union is over $v\in t_{\mu}W_0$ and $w\in W_0^{\la}m_{\la}$. Therefore the bijection in \cite{ramparkinsonschwer} gives a bijection between $K$ cosets in $U^-t_{\mu}K\cap Kt_{\la}K$ and labelled positively folded alcove walks in~$\cP(\vec\la)_{\mu}$. This bijection between $K$ cosets is the same as that which follows from the results in~\cite{litt}.

\smallskip

\noindent\textbf{Theorem~$\mathbf{R}$ $\Longleftrightarrow$ Theorem $\mathbf{C}'$.} The Weyl character is $P_\lambda(x;0) = s_\lambda(x)
= \sum_\mu \dim(V(\lambda)_\mu) x^\mu$.  Specialising (\ref{eq:Pladefn}) at $q^{-1}=0$ gives a path formula for $s_\lambda(x)$,
and only those paths with dimension $\dim(p) = \langle \lambda+\mu,\rho\rangle$ survive the specialisation.
\end{proof}


\section{Alcove walks}\label{section:alcovewalks}

We begin this section by recalling the standard geometric interpretation of affine Weyl groups from~\cite{bourbaki}. We then define positively folded alcove walks and root operators following~\cite{litt} and~\cite{ram2}. 
A few elementary lemmas and propositions provide the background for the proof of Theorem~$\mathbf{C}'$ in the next section.

\subsection{The geometry of affine Weyl groups}\label{subsect:geom}

Let $R$ be a reduced irreducible root system in an $n$-dimensional real vector space $\fh_{\RR}$ with inner product $\langle\cdot,\cdot\rangle$. For $\alpha\in\fh_{\RR}\backslash\{0\}$ let $\alpha^{\vee}=\frac{2\alpha}{\langle\alpha,\alpha\rangle}$. Let $\{\alpha_1,\ldots,\alpha_n\}$ be a set of simple roots of~$R$, and let $R^+$ be the associated set of positive roots. The \textit{coroot lattice} is $Q=\ZZ\alpha_1^{\vee}+\cdots+\ZZ\alpha_n^{\vee}$, and the coweight lattice is $P=\ZZ\omega_1+\cdots+\ZZ\omega_n$, where $\{\omega_1,\ldots,\omega_n\}$ is the basis of $\fh_{\RR}$ defined by $\langle\omega_i,\alpha_j\rangle=\delta_{ij}$ for $1\leq i,j\leq n$. The set of \textit{dominant coweights} is $P^+=\ZZ_{\geq0}\omega_1+\cdots+\ZZ_{\geq0}\omega_n$. There is a unique \textit{highest root} $\theta$ of $R$ satisfying $\langle\omega_i,\theta\rangle\geq\langle\omega_i,\alpha\rangle$ for all $1\leq i\leq n$ and all $\alpha\in R$.

For each $\alpha\in R$ define a linear hyperplane $H_{\alpha}=\{\la\in\fh_{\RR}\mid\langle\la,\alpha\rangle=0\}$. The orthogonal reflection in the hyperplane $H_{\alpha}$ is $s_{\alpha}(\la)=\la-\langle\la,\alpha\rangle\alpha^{\vee}$, and the \textit{Weyl group} $W_0$ of $R$ is the subgroup of $GL(\fh_{\RR})$ generated by $\{s_{\alpha}\mid\alpha\in R\}$. The Weyl group is a finite Coxeter group with distinguished generators $s_1,\ldots,s_n$ (where $s_i=s_{\alpha_i}$) and thus has a \textit{length function} $\ell:W_0\to\ZZ_{\geq0}$, with $\ell(w)$ being the smallest $\ell\geq0$ such that $w=s_{i_1}\cdots s_{i_{\ell}}$. Let $w_0$ be the longest element of~$W_0$.

The open connected components of $\fh_{\RR}\backslash\bigcup_{\alpha\in R}H_{\alpha}$ are \textit{Weyl chambers} (or \textit{Weyl sectors}). These are open simplicial cones, and $W_0$ acts simply transitively on the set of Weyl chambers. The \textit{fundamental Weyl chamber} is
$
C_0=\{\la\in\fh_{\RR}\mid\langle\la,\alpha_i\rangle>0\textrm{ for }i=1,\ldots,n\},
$
and $P^+=P\cap\overline{C_0}$, where $\overline{C_0}$ is the closure of $C_0$ in $\fh_\RR$.

The roots $\alpha\in R$ can be regarded as elements of $\fh_{\RR}^*$ by setting $\alpha(\la)=\langle\la,\alpha\rangle$ for~$\la\in\fh_{\RR}$. Let $\delta:\fh_{\RR}\to\RR$ be the (non-linear) constant function with $\delta(\la)=1$ for all $\la\in\fh_{\RR}$. The \textit{affine root system} is $R_a=R+\ZZ\delta$. The \textit{affine hyperplane} for the affine root $\alpha+j\delta$ is
$$
H_{\alpha+j\delta}=\{\la\in\fh_{\RR}\mid\langle\la,\alpha+j\delta\rangle=0\}=\{\la\in\fh_{\RR}\mid\langle\la,\alpha\rangle=-j\}=H_{-\alpha-j\delta}.
$$
The \textit{affine Weyl group} is the subgroup $W$ of $\mathrm{Aff}(\fh_{\RR})$ generated by the reflections $s_{\alpha+k\delta}$ with $\alpha+k\delta\in R_a$, where $s_{\alpha+k\delta}:\fh_{\RR}\to\fh_{\RR}$ is given by
$
s_{\alpha+k\delta}(\la)=\la-(\langle\la,\alpha\rangle+k)\alpha^{\vee}
$ for $\la\in\fh_{\RR}$.
Let $\alpha_0=-\theta+\delta$ (with $\theta$ the highest root of $R$). The affine Weyl group is a Coxeter group with distinguished generators $s_0,s_1,\ldots,s_n$, where $s_0=s_{\alpha_0}$. For $\mu\in\fh_{\RR}$, let $t_{\mu}:\fh_{\RR}\to\fh_{\RR}$ be the translation $t_{\mu}(\la)=\la+\mu$ for all $\la\in\fh_{\RR}$. Then $s_{\alpha+k\delta}=t_{-k\alpha^{\vee}}s_{\alpha}$ and $W$ is the semidirect product $W=Q\rtimes W_0$.

The open connected components of $\fh_{\RR}\backslash\bigcup_{\beta\in R_a}H_{\beta}$ are \textit{alcoves}. The \textit{fundamental alcove} is
$$
c_0=\{\la\in\fh_{\RR}\mid\langle\la,\alpha_i\rangle>0\textrm{ for all $i=0,\ldots,n$}\}\subset C_0.
$$
The affine Weyl group acts simply transitively on the set of alcoves, and therefore $W$ is in bijection with the set of alcoves. Identify $1$ with $c_0$. 

The \textit{extended affine Weyl group} $\tilde{W}=P\rtimes W_0$ acts transitively (but in general not simply transitively) on the set of alcoves. In general $\tilde{W}$ is not a Coxeter group, but it is ``nearly'' a Coxeter group: There is a length function $\ell:\tilde{W}\to\ZZ_{\geq0}$ defined by 
$$
\ell(w)=|\{H_{\alpha+j\delta}\mid \textrm{$H_{\alpha+j\delta}$ separates $c_0$ from $wc_0$}\}|,
$$
and for $w\in W\subseteq \tilde{W}$ this agrees with the Coxeter length function. If $\Omega=\{w\in\tilde{W}\mid\ell(w)=0\}$ then $\tilde{W}=W\rtimes\Omega$, and $\Omega$ is isomorphic to the finite abelian group~$P/Q$. Therefore $\tilde{W}$ acts simply transitively on the set of alcoves in $\fh_{\RR}\times\Omega$, and so $\tilde{W}$ can be thought of, geometrically, as $|\Omega|$ copies of $W$.

\begin{Example}\label{ex:picture}  Let $e_1,e_2,e_3$ be the standard basis of $\RR^3$,
let $\fh_{\RR}=\{x\in\RR^3\mid x_1+x_2+x_3=0\}$, and let $R=\{\alpha_1,\alpha_2,\theta\}$, where $\alpha_1=e_1-e_2$, $\alpha_2=e_2-e_3$, and $\theta=\alpha_1+\alpha_2=e_1-e_3$. There are $6$ Weyl chambers (the $6$ sectors based at $0$) in bijection with~$W_0\cong S_3$ (the symmetric group on $3$ letters). The alcoves of $W$ are the triangles, which are in bijection with~$W$ (with $c_0\leftrightarrow 1$). The coroots $\alpha_1^{\vee}$, $\alpha_2^{\vee}$ and $\theta^{\vee}$ are shown, and the coroot lattice $Q=\ZZ\alpha_1^{\vee}+\ZZ\alpha_2^{\vee}$ is the set of centres of the solid hexagons. This ``hexagonification'' makes the semidirect product structure $W=Q\rtimes W_0$ clear. The fundamental coweights $\omega_1$ and $\omega_2$ are shown. In this case the vertices of $c_0$ are $\{0,\omega_1,\omega_2\}$ (see Appendix~\ref{app:buildings} for the general case), and the coweight lattice $P$ is the set of all vertices. The dominant coweights are the elements of $P$ in the ``top $\frac{1}{6}$th'' of~$\fh_{\RR}$. We have $\Omega\cong P/Q\cong\ZZ/3\ZZ$, and so $\tilde{W}$ can be thought of as $3$ copies (sheets) of the picture below. The ``orientation'' on the hyperplanes is explained at the beginning of Section~\ref{subsect:alcovewalk}.
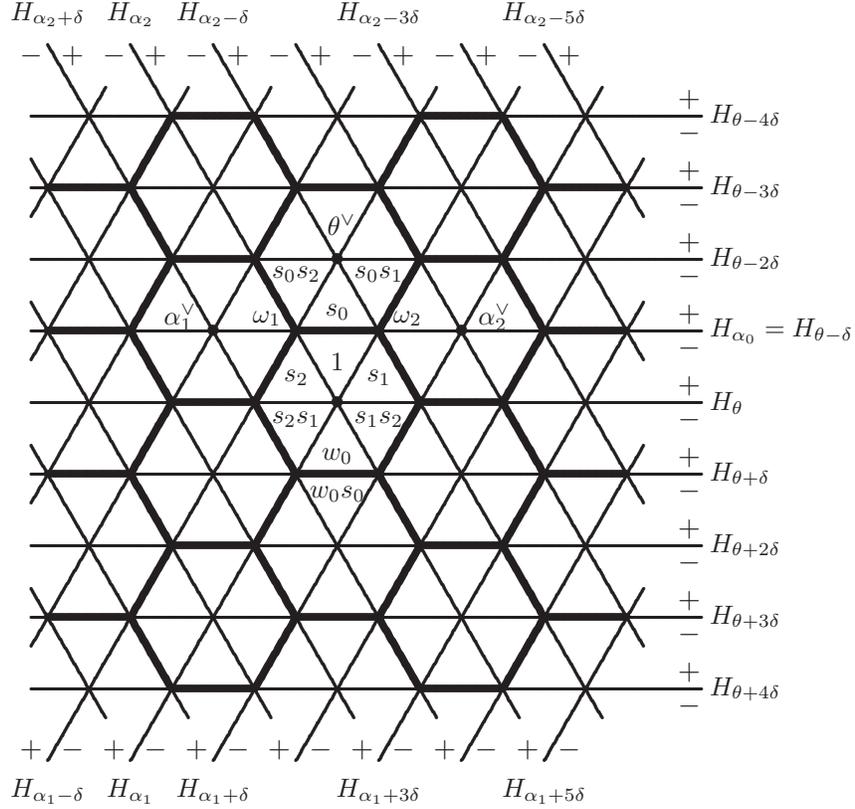
$$
\beginpicture
\setcoordinatesystem units <1.1cm,1.1cm>         
\setplotarea x from -4 to 5, y from -5 to 5.2  
    \put{\small{$1$}} at 0 0.5
    \put{\small{$s_0$}} at 0 1.1
    \put{\small{$s_1$}} at 0.5 0.3
    \put{\small{$s_2$}} at -0.5 0.3
    \put{\small{$s_0s_1$}} at 0.5 1.55
    \put{\small{$s_0s_2$}} at -0.5 1.55
    \put{\small{$s_1s_2$}} at 0.5 -0.2
    \put{\small{$s_2s_1$}} at -0.5 -0.2
    \put{\small{$w_0$}} at 0 -0.65
    \put{\small{$w_0s_0$}} at 0 -1.1
    \plot 2.5 -4.33 3.7 -2.2516 /
    \plot 1.5 -4.33 3.7 -0.5196 /
    \plot 0.5 -4.33 3.7 1.2124 /
    \plot -0.5 -4.33 3.7 2.9444 /
    \plot -1.5 -4.33 3.2 3.81 /
    \plot -2.5 -4.33 2.2 3.81 /
    \plot -3.5 -4.33 1.2 3.81 /
    \plot -3.7 -2.9444 0.2 3.81 /
    \plot -3.7 -1.2124 -0.8 3.81 /
    \plot -3.7 0.5196 -1.8 3.81 /
    \plot -3.7 2.2516 -2.8 3.81 /
    \plot 2.5 4.33 3.7 2.2516 /
    \plot 1.5 4.33 3.7 0.5196 /
    \plot 0.5 4.33 3.7 -1.2124 /
    \plot -0.5 4.33 3.7 -2.9444 /
    \plot -1.5 4.33 3.2 -3.81 /
    \plot -2.5 4.33 2.2 -3.81 /
    \plot -3.5 4.33 1.2 -3.81 /
    \plot -3.7 2.9444 0.2 -3.81 /
    \plot -3.7 1.2124 -0.8 -3.81 /
    \plot -3.7 -0.5196 -1.8 -3.81 /
    \plot -3.7 -2.2516 -2.8 -3.81 /
    \plot -3.7 -3.464 4.4 -3.464 /
    \plot -3.7 -2.598 4.4 -2.598 /
    \plot -3.7 -1.732 4.4 -1.732 /
    \plot -3.7 -0.866 4.4 -0.866 /
    \plot -3.7 0 4.4 0 /
    \plot -3.7 3.464 4.4 3.464 /
    \plot -3.7 2.598 4.4 2.598 /
    \plot -3.7 1.732 4.4 1.732 /
    \plot -3.7 0.866 4.4 0.866 /
    \put{\small{$H_{\alpha_1-\delta}$}} at -3.5 -4.7
    \put{\small{$H_{\alpha_1}$}} at -2.5 -4.7
    \put{\small{$H_{\alpha_1+\delta}$}} at -1.5 -4.7
    \put{\small{$H_{\alpha_1+3\delta}$}} at 0.5 -4.7
    \put{\small{$H_{\alpha_1+5\delta}$}} at 2.5 -4.7
    \put{\small{$H_{\alpha_2+\delta}$}} at -3.5 4.7
    \put{\small{$H_{\alpha_2}$}} at -2.5 4.7
    \put{\small{$H_{\alpha_2-\delta}$}} at -1.5 4.7
    \put{\small{$H_{\alpha_2-3\delta}$}} at 0.5 4.7
    \put{\small{$H_{\alpha_2-5\delta}$}} at 2.5 4.7
    \put{\small{$H_{\theta-4\delta}$}}[l] at 4.5 3.464
    \put{\small{$H_{\theta-3\delta}$}}[l] at 4.5 2.598
    \put{\small{$H_{\theta-2\delta}$}}[l] at 4.5 1.732
    \put{\small{$H_{\alpha_0}=H_{\theta-\delta}$}}[l] at 4.5 0.866
    \put{\small{$H_{\theta}$}}[l] at 4.5 0
    \put{\small{$H_{\theta+\delta}$}}[l] at 4.5 -0.866
    \put{\small{$H_{\theta+2\delta}$}}[l] at 4.5 -1.732
    \put{\small{$H_{\theta+3\delta}$}}[l] at 4.5 -2.598
    \put{\small{$H_{\theta+4\delta}$}}[l] at 4.5 -3.464
    \put{$+$} at 4.25 0.2
    \put{$+$} at 4.25 1.066
    \put{$+$} at 4.25 1.932
    \put{$+$} at 4.25 2.798
    \put{$+$} at 4.25 3.664
    \put{$-$} at 4.25 -1.066
    \put{$-$} at 4.25 -1.932
    \put{$-$} at 4.25 -2.798
    \put{$-$} at 4.25 -3.664
    \put{$+$} at 4.25 -3.264
    \put{$+$} at 4.25 -2.398
    \put{$+$} at 4.25 -1.532
    \put{$+$} at 4.25 -0.666
    \put{$-$} at 4.25 3.264
    \put{$-$} at 4.25 2.398
    \put{$-$} at 4.25 1.532
    \put{$-$} at 4.25 0.666
    \put{$-$} at 4.25 -0.2
    \put{$+$} at -3.7 -4.2
    \put{$-$} at -3.2 -4.2
    \put{$+$} at -2.7 -4.2
    \put{$-$} at -2.2 -4.2
    \put{$+$} at -1.7 -4.2
    \put{$-$} at -1.2 -4.2
    \put{$+$} at -0.7 -4.2
    \put{$-$} at -0.2 -4.2
    \put{$+$} at 0.3 -4.2
    \put{$-$} at 0.8 -4.2
    \put{$+$} at 1.3 -4.2
    \put{$-$} at 1.8 -4.2
    \put{$+$} at 2.3 -4.2
    \put{$-$} at 2.8 -4.2
    \put{$-$} at -3.7 4.2
    \put{$+$} at -3.2 4.2
    \put{$-$} at -2.7 4.2
    \put{$+$} at -2.2 4.2
    \put{$-$} at -1.7 4.2
    \put{$+$} at -1.2 4.2
    \put{$-$} at -0.7 4.2
    \put{$+$} at -0.2 4.2
    \put{$-$} at 0.3 4.2
    \put{$+$} at 0.8 4.2
    \put{$-$} at 1.3 4.2
    \put{$+$} at 1.8 4.2
    \put{$-$} at 2.3 4.2
    \put{$+$} at 2.8 4.2
    \put{\small{$\omega_2$}} at 0.85 1.02
    \put{\small{$\omega_1$}} at -0.85 1.02
    \put{\small{$\alpha_1^{\vee}$}} at -1.9 1.03
    \put{$\bullet$} at -1.5 0.866
    \put{\small{$\alpha_2^{\vee}$}} at 1.9 1.03
    \put{$\bullet$} at 1.5 0.866
    \put{\small{$\theta^{\vee}$}} at 0.05 2.15
    \put{$\bullet$} at 0 1.732
    \put{$\bullet$} at 0 0 
    \setplotsymbol({\tiny{$\bullet$}})
    \plot 1 -3.464 0.5 -2.598 1 -1.732 0.5 -0.866 1 0 0.5 0.866 1 1.732 0.5 2.598 1 3.464 /
    \plot 2 -3.464 2.5 -2.598 2 -1.732 2.5 -0.866 2 0 2.5 0.866 2 1.732 2.5 2.598 2 3.464 /
    \plot -1 -3.464 -0.5 -2.598 -1 -1.732 -0.5 -0.866 -1 0 -0.5 0.866 -1 1.732 -0.5 2.598 -1 3.464 /
    \plot -2 -3.464 -2.5 -2.598 -2 -1.732 -2.5 -0.866 -2 0 -2.5 0.866 -2 1.732 -2.5 2.598 -2 3.464 /
    \plot -0.5 0.866 0.5 0.866 /
    \plot -0.5 2.598 0.5 2.598 /
    \plot -0.5 -0.866 0.5 -0.866 /
    \plot -0.5 -2.598 0.5 -2.598 /
    \plot 1 -3.464 2 -3.464 /
    \plot 1 -1.732 2 -1.732 /
    \plot 1 0 2 0 /
    \plot 1 1.732 2 1.732 /
    \plot 1 3.464 2 3.464 /
    \plot -1 -3.464 -2 -3.464 /
    \plot -1 -1.732 -2 -1.732 /
    \plot -1 0 -2 0 /
    \plot -1 1.732 -2 1.732 /
    \plot -1 3.464 -2 3.464 /
    \plot 2.5 -2.598 3.5 -2.598 /
    \plot 2.5 -0.866 3.5 -0.866 /
    \plot 2.5 2.598 3.5 2.598 /
    \plot 2.5 0.866 3.5 0.866 /
    \plot -2.5 -2.598 -3.5 -2.598 /
    \plot -2.5 -0.866 -3.5 -0.866 /
    \plot -2.5 2.598 -3.5 2.598 /
    \plot -2.5 0.866 -3.5 0.866 /
\endpicture
$$
\end{Example}

\subsection{Alcove walks}\label{subsect:alcovewalk}

Each hyperplane $H$ determines two closed halfspaces of $\fh_{\RR}$. Define an \textit{orientation} on the linear hyperplanes $H_{\alpha}$, $\alpha\in R$, by declaring the ``positive'' side of $H_{\alpha}$ to be the half space containing the fundamental Weyl chamber~$C_0$. Extend this orientation ``periodically" by giving $H_{\alpha+k\delta}$ the same orientation as~$H_{\alpha}$. An equivalent definition is that the positive side of $H_{\alpha+k\delta}$ is the half space which contains a ``subsector'' of $C_0$. Explicitly, if $\alpha\in R^+$ and $k\in\ZZ$ then the negative and positive sides of $H_{\alpha+k\delta}$ are
\begin{align*}
H_{\alpha+k\delta}^-&=\{x\in\fh_{\RR}\mid\langle x,\alpha+k\delta\rangle\leq0\}=\{x\in\fh_{\RR}\mid \langle x,\alpha\rangle\leq -k\},\quad\textrm{and}\\
H_{\alpha+k\delta}^+&=\{x\in\fh_{\RR}\mid\langle x,\alpha+k\delta\rangle\geq0\}=\{x\in\fh_{\RR}\mid\langle x,\alpha\rangle\geq -k\},
\end{align*}
respectively. See the picture in Example~\ref{ex:picture}. For $v\in\tilde{W}$ the \textit{signed length} of $v$ is
\begin{align}\label{eq:sgnlength}
\epsilon(v)=\#\left\{\textrm{$H$ with }\beginpicture
\setcoordinatesystem units <0.8cm,0.8cm>         
\setplotarea x from -0.8 to 0.7, y from -0.5 to 0.5  
\put{$\scriptstyle{H}$}[b] at 0 0.6
\put{$\scriptstyle{-}$}[b] at -0.4 0.25
\put{$\scriptstyle{+}$}[b] at 0.4 0.25
\put{$\scriptstyle{1}$}[br] at -0.45 -0.1
\put{$\scriptstyle{v}$}[bl] at 0.45 -0.1
\plot  0 -0.4  0 0.5 /
\endpicture\right\}-\#\left\{\textrm{$H$ with }\beginpicture
\setcoordinatesystem units <0.8cm,0.8cm>         
\setplotarea x from -0.8 to 0.7, y from -0.5 to 0.5  
\put{$\scriptstyle{H}$}[b] at 0 0.6
\put{$\scriptstyle{-}$}[b] at -0.4 0.25
\put{$\scriptstyle{+}$}[b] at 0.4 0.25
\put{$\scriptstyle{v}$}[br] at -0.45 -0.1
\put{$\scriptstyle{1}$}[bl] at 0.45 -0.1
\plot  0 -0.4  0 0.5 /
\endpicture\right\},
\qquad\hbox{so that}
\end{align}
\begin{align}\label{eq:epsilformula}
\epsilon(t_{\mu})=\langle\mu,2\rho\rangle,\ \ \hbox{for $\mu\in P$},
\qquad\hbox{and}\qquad
\epsilon(t_{\la})=\ell(t_{\la}),\ \ \hbox{for $\lambda\in P^+$,}
\end{align}
where $\rho = \frac{1}{2} \sum_{\alpha\in R^+} \alpha$.

Let $\vec w=s_{i_1}\cdots s_{i_{\ell}}\gamma$ be a reduced expression for $w\in \tilde{W}$, with $\gamma\in\Omega$. A \textit{positively folded alcove walk of type~$\vec w$} is a sequence of steps from alcove to alcove in $\tilde{W}$, starting at $1\in \tilde{W}$, and made up of the symbols
\begin{align}\label{eq:symbols}
\beginpicture
\setcoordinatesystem units <0.8cm,0.8cm>         
\setplotarea x from -0.8 to 0.7, y from -0.5 to 0.5  
\put{$\scriptstyle{-}$}[b] at -0.4 0.25
\put{$\scriptstyle{+}$}[b] at 0.4 0.25
\put{$\scriptstyle{x}$}[br] at -0.6 0.1
\put{$\scriptstyle{xs}$}[bl] at 0.6 0.1
\plot  0 -0.4  0 0.5 /
\arrow <5pt> [.2,.67] from -0.5 0 to 0.5 0   %
\put{(\textit{positive crossing})} at 0 -1
\endpicture
\qquad\qquad
\beginpicture
\setcoordinatesystem units <0.8cm,0.8cm>         
\setplotarea x from -0.8 to 0.7, y from -0.5 to 0.5  
\put{$\scriptstyle{-}$}[b] at -0.4 0.35
\put{$\scriptstyle{+}$}[b] at 0.4 0.35
\put{$\scriptstyle{x}$}[bl] at 0.6 0.1
\put{$\scriptstyle{xs}$}[br] at -0.6 0.1
\plot  0 -0.4  0 0.6 /
\plot 0.5 0  0.05 0 /
\arrow <5pt> [.2,.67] from 0.05 0.1 to 0.5 0.1   %
\plot 0.05 0 0.05 0.1 /
\put{(\textit{positive fold})} at 0 -1
\endpicture
\qquad\qquad
\beginpicture
\setcoordinatesystem units <0.8cm,0.8cm>         
\setplotarea x from -0.8 to 0.7, y from -0.5 to 0.5  
\put{$\scriptstyle{-}$}[b] at -0.4 0.25
\put{$\scriptstyle{+}$}[b] at 0.4 0.25
\put{$\scriptstyle{x}$}[bl] at 0.6 0.1
\put{$\scriptstyle{xs}$}[br] at -0.6 0.1
\plot  0 -0.4  0 0.5 /
\arrow <5pt> [.2,.67] from 0.5 0 to -0.5 0   %
\put{(\textit{negative crossing})} at 0 -1 
\endpicture
\end{align}
where the $k$th step has $s=s_{i_k}$ for $k=1,\ldots,\ell$. To take into account the sheets of $\tilde{W}$, one concludes the alcove walk by ``jumping'' to the $\gamma$ sheet of $\fh_{\RR}\times\Omega$. Our pictures of alcove walks will always be drawn without this jump by projecting $\fh_{\RR}\times\{\gamma\}\rightarrow\fh_{\RR}\times \{1\}$.

Let $p$ be a positively folded alcove walk. Define statistics
\begin{align*}
\epsilon^+(p)&=\#\textrm{(positive crossings in $p$)},\\
\epsilon^-(p)&=\#\textrm{(negative crossings in $p$)},\\
f(p)&=\#\textrm{(folds in $p$)}.
\end{align*}
The \textit{length} $\ell(p)$ of $p$ and the \textit{signed length} $\epsilon(p)$ of $p$ are
$$
\ell(p)=\epsilon^+(p)+\epsilon^-(p)+f(p)\qquad\textrm{and}\qquad 
\epsilon(p)=\epsilon^+(p)-\epsilon^-(p)= \epsilon(\mathrm{end}(p)),
$$
where $\mathrm{end}(p)\in\tilde{W}$ is the alcove where $p$ ends, and $\epsilon(\mathrm{end}(p))$ is as in~(\ref{eq:sgnlength}). The equality $\epsilon(p)=\epsilon(\mathrm{end}(p))$ follows because if $p$ makes a positive and a negative crossing on the same hyperplane then the corresponding contributions to $\epsilon(p)=\epsilon^+(p)-\epsilon^-(p)$ cancel. The \textit{dimension} of $p$ is
\begin{align}\label{eq:dimension}
\dim(p)=\epsilon^+(p)+f(p)=\ell(p)-\epsilon^-(p)=\hbox{$\frac12$}(\ell(p)+\epsilon(p)+f(p)).
\end{align}

\begin{Example}\label{ex:1} The positively folded alcove walk 
$$\beginpicture
\setcoordinatesystem units <0.8cm,0.8cm>         
\setplotarea x from -4 to 4, y from -3.1 to 3.1  
    \plot 4.3 -2.944 5.2 -1.3856 /
    \plot 3.3 -2.944 5.2 0.346 /
    \plot 2.3 -2.944 5.2 2.078 /
    \plot 1.3 -2.944 4.7 2.944 /
    \plot 0.3 -2.944 3.7 2.944 /
    \plot -0.7 -2.944 2.7 2.944 /
    \plot -1.7 -2.944 1.7 2.944 /
    \plot -2.7 -2.944 0.7 2.944 /
    \plot -3.7 -2.944 -0.3 2.944 /
    \plot -3.7 -1.212 -1.3 2.944 /
    \plot -3.7 0.520 -2.3 2.944 /
    \plot -3.7 2.2516 -3.3 2.944 /
    \plot 4.3 2.944 5.2 1.3856 /
    \plot 3.3 2.944 5.2 -0.346 /
    \plot 2.3 2.944 5.2 -2.078 /
    \plot 1.3 2.944 4.7 -2.944 /
    \plot 0.3 2.944 3.7 -2.944 /
    \plot -0.7 2.944 2.7 -2.944 /
    \plot -1.7 2.944 1.7 -2.944 /
    \plot -2.7 2.944 0.7 -2.944 /
    \plot -3.7 2.944 -0.3 -2.944 /
    \plot -3.7 1.212 -1.3 -2.944 /
    \plot -3.7 -0.520 -2.3 -2.944 /
    \plot -3.7 -2.2516 -3.3 -2.944 /
    \plot -3.7 -2.598 5.2 -2.598 /
    \plot -3.7 -1.732 5.2 -1.732 /
    \plot -3.7 -0.866 5.2 -0.866 /
    \plot -3.7 0 5.2 0 /
    \plot -3.7 2.598 5.2 2.598 /
    \plot -3.7 1.732 5.2 1.732 /
    \plot -3.7 0.866 5.2 0.866 /
    \arrow <5pt> [.2,.67] from -2 0.577 to -1.5 0.289   %
    \arrow <5pt> [.2,.67] from -1.5 0.289 to -1.5 -0.289   %
    \arrow <5pt> [.2,.67] from -1.5 -0.289 to -1 -0.577   %
    \arrow <5pt> [.2,.67] from -1 -0.577 to -1 -1.082   %
    \put{\beginpicture
    \setcoordinatesystem units <0.75cm,0.75cm>         
    \plot 0 0 -0.28 -0.15 /
    \plot -0.28 -0.15 -0.2 -0.3 /
    \arrow <5pt> [.2,.67] from -0.2 -0.3 to 0.05 -0.13
    \endpicture} at -1 -1.082
    \arrow <5pt> [.2,.67] from -0.95 -1.155 to -0.55 -1.432   %
    \put{\beginpicture
    \setcoordinatesystem units <0.7cm,0.7cm>         
    \arrow <5pt> [.2,.67] from 0.08 -0.32 to 0.08 0 %
    \plot 0.07 -0.32 -0.08 -0.32 /
    \plot -0.08 -0.32 -0.08 0 /
    \endpicture} at -0.5 -1.417
    \arrow <5pt> [.2,.67] from -0.45 -1.432 to 0 -1.155   %
    \arrow <5pt> [.2,.67] from 0 -1.155 to 0.5 -1.443   %
    \arrow <5pt> [.2,.67] from 0.5 -1.443 to 1 -1.155   %
    \arrow <5pt> [.2,.67] from 1 -1.155 to 1.5 -1.443   %
    \arrow <5pt> [.2,.67] from 1.5 -1.443 to 2 -1.155   %
    \arrow <5pt> [.2,.67] from 2 -1.155 to 2.5 -1.443   %
    \arrow <5pt> [.2,.67] from 2.5 -1.443 to 2.97 -1.23   %
    \put{\beginpicture
    \setcoordinatesystem units <0.7cm,0.7cm>         
    \arrow <5pt> [.2,.67] from 0.28 -0.15 to 0 0
    \plot 0.28 -0.15 0.2 -0.3 /
    \plot 0.2 -0.3 -0.05 -0.13 /
    \endpicture} at 3 -1.155
    \arrow <5pt> [.2,.67] from 3 -1.13 to 3 -0.577   %
    \arrow <5pt> [.2,.67] from 3 -0.577 to 2.5 -0.289   %
    \arrow <5pt> [.2,.67] from 2.5 -0.289 to 2.5 0.289   %
    \arrow <5pt> [.2,.67] from 2.5 0.289 to 2 0.577   %
    \arrow <5pt> [.2,.67] from 2 0.577 to 2 1.155   %
\endpicture$$
has $\epsilon^+(p)=9$, $\epsilon^-(p)=8$, $f(p)=3$, $\ell(p)=20$, $\epsilon(p)=1$ and $\dim(p)=12$.
\end{Example}

As in the statement of Theorem~$\mathbf{C}$, for $\la\in P^+$ and $\mu\in P$ let
$$
\cP(\vec\la)_{\mu}=\left\{\begin{matrix}\textrm{positively folded alcove walks of type}\\
\textrm{$\vec u\cdot\vec m_{\la}$ with end alcove in $t_{\mu}W_0$}\end{matrix} \,\,\bigg|\,\, u\in W_0^{\la}\right\},
$$
where $\vec m_{\la}$ is a minimal length walk to the minimal length element $m_{\la}\in W_0t_{\la}W_0$, and $\vec u$ is a minimal length walk to $u\in W_0^{\la}$, with $W_0^{\la}$ a set of minimal length representatives for cosets in $W_0/W_{0\la}$. Note that if $\la\in P^+$ is \textit{regular} (that is, $\la$ does not lie on any hyperplane $H_{\alpha}$, $\alpha\in R$) then $W_0^{\la}=W_0$.

Let $p\in\cP(\vec\la)_{\mu}$. The \textit{final direction} of $p$ is the element $\varphi(p)\in W_0$ defined by $\mathrm{end}(p)=t_{\mu}\varphi(p)$. Since $t_{\mu}$ is in the $1$-position of $t_{\mu}W_0$ we have
\begin{align}\label{eq:signedlength}
\epsilon(p)=\epsilon(t_{\mu}\varphi(p))=\epsilon(t_{\mu})-\ell(\varphi(p))=\langle\mu,2\rho\rangle-\ell(\varphi(p)).
\end{align}
The \textit{canonical antidominant $\vec\la$-path} is the unique element 
$p_{w_0\la}\in\cP(\vec\la)_{w_0\la}$, where $w_0$ is the longest element of $W_0$. 
This walk consists entirely of negative crossings, has no folds, and has 
$\mathrm{end}(p_{w_0\la})=t_{w_0\la}$ and $\dim(p_{w_0\la})=0$.

\begin{Example}\label{ex:paths} Consider type $A_2$ and let $\la=\omega_1+\omega_2$. Then $\vec m_{\la}=s_0$, and $W_0^{\la}=W_0$. The $25$ positively folded alcove walks of types $\vec u\cdot \vec m_{\la}$ with $u\in W_0^{\la}$ are shown below. Within each choice for $u$ the paths have been organised according to end vertex $\mu\in\{0\}\cup W_0\la$.
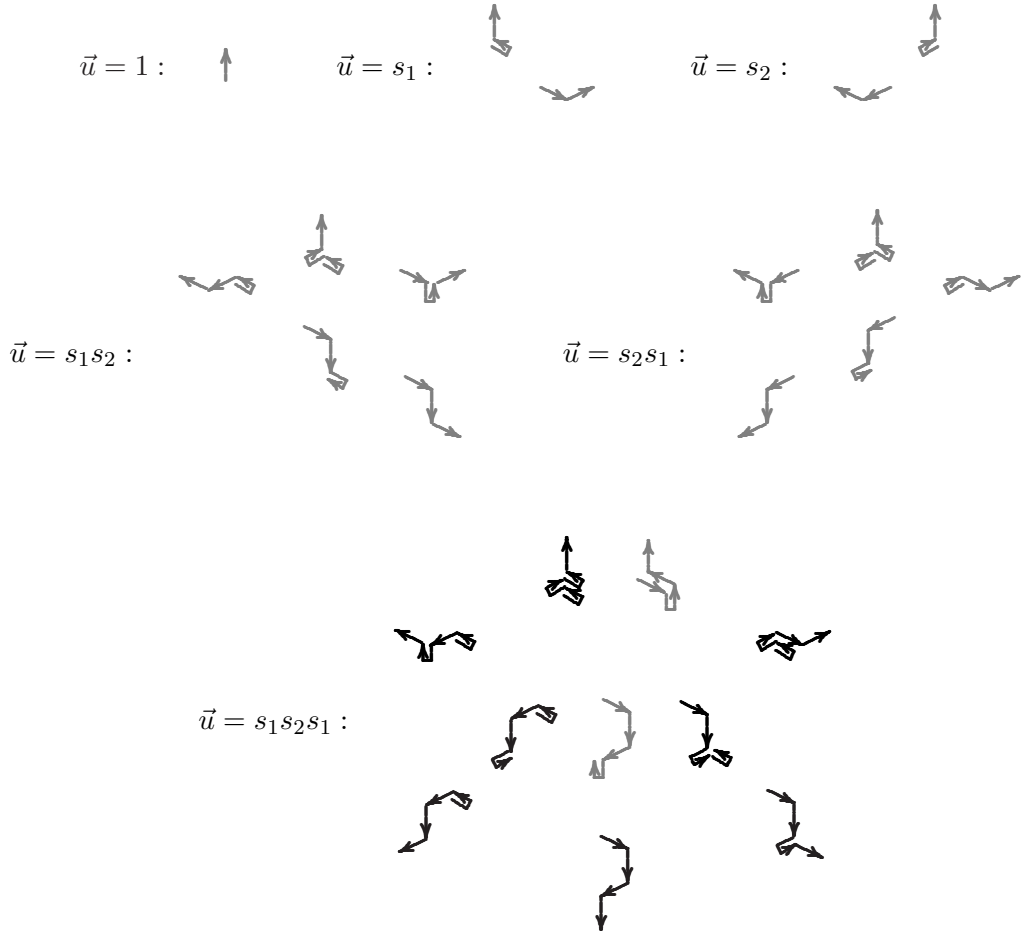
$$
\vec u=1:\beginpicture
    \setcoordinatesystem units <0.7cm,0.7cm>         
    \setplotarea x from -1 to 1, y from -0.4 to 0.8  
        \color{Gray}
        \put{\ua} at 0 0.2
        \color{black}
\endpicture
\qquad
\vec u=s_1:\beginpicture
    \setcoordinatesystem units <1.2cm,1.2cm>
    \setplotarea x from -1 to 1, y from -0.4 to 0.8 
	\color{Gray}
	\put{\beginpicture
    \setcoordinatesystem units <0.7cm,0.7cm>         
    \setplotarea x from -0.7 to 0.7, y from -0.7 to 0.7  
        \put{\ra} at 0 0
        \put{\upra} at 0.5 0
    \endpicture} at 0.3 -0.2
    \put{\beginpicture
    \setcoordinatesystem units <0.7cm,0.7cm>         
    \setplotarea x from -0.7 to 0.7, y from -1 to 1  
        \put{\rfold} at 0 0
        \put{\ua} at 0 0.35
    \endpicture} at -0.5 0.5
    \color{black}
\endpicture
\qquad
\vec u=s_2:\beginpicture
    \setcoordinatesystem units <1.2cm,1.2cm>
    \setplotarea x from -1 to 1, y from -0.4 to 0.8 
	\color{Gray}
	\put{\beginpicture
    \setcoordinatesystem units <0.7cm,0.7cm>         
    \setplotarea x from -0.7 to 0.7, y from -0.7 to 0.7  
        \put{\lfoldalt} at 0 0
        \put{\ua} at 0 0.35
    \endpicture} at 0.5 0.5
    \put{\beginpicture
    \setcoordinatesystem units <0.7cm,0.7cm>         
    \setplotarea x from -0.7 to 0.7, y from -0.7 to 0.7  
        \put{\nla} at 0 0
        \put{\upla} at -0.55 0
    \endpicture} at -0.3 -0.2
    \color{black} 
\endpicture
$$
$$
\vec u= s_1s_2:\beginpicture
    \setcoordinatesystem units <1.2cm,1.2cm>
    \setplotarea x from -2 to 2, y from -1 to 1.5 
	\color{Gray}
	\put{\beginpicture
    \setcoordinatesystem units <0.7cm,0.7cm>         
    \setplotarea x from -0.7 to 0.7, y from -0.7 to 0.7  
        \put{\rfold} at 0 0
        \put{\lfoldalt} at -0.1 0.16
        \put{\ua} at -0.1 0.5
    \endpicture} at 0 1.3
    \put{\beginpicture
    \setcoordinatesystem units <0.7cm,0.7cm>         
    \setplotarea x from -0.7 to 0.7, y from -0.7 to 0.7  
        \put{\ra} at 0 0
        \put{\dfolddash} at 0.3 -0.15
        \put{\upra} at 0.7 0
    \endpicture} at 1.2 0.9
    \put{\beginpicture
    \setcoordinatesystem units <0.7cm,0.7cm>         
    \setplotarea x from -0.7 to 0.7, y from -0.7 to 0.7  
        \put{\rfold} at 0 0
        \put{\nla} at -0.3 -0.12
        \put{\upla} at -0.85 -0.1
    \endpicture} at -1.2 0.9
    \put{\beginpicture
    \setcoordinatesystem units <0.7cm,0.7cm>         
    \setplotarea x from -0.7 to 0.7, y from -0.7 to 0.7  
        \put{\ra} at 0 0
        \put{\da} at 0.25 -0.45
        \put{\ra} at 0.53 -0.9
    \endpicture} at 1.2 -0.5
    \put{\beginpicture
    \setcoordinatesystem units <0.7cm,0.7cm>         
    \setplotarea x from -0.7 to 0.7, y from -0.7 to 0.7  
        \put{\ra} at 0 0
        \put{\da} at 0.25 -0.45
        \put{\rfoldalt} at 0.25 -0.75
    \endpicture} at 0 0.1
     \color{black}
    \endpicture
    \qquad
    \vec u= s_2s_1:
	\beginpicture
    \setcoordinatesystem units <1.2cm,1.2cm>
    \setplotarea x from -2 to 2, y from -1 to 1.5 
	\color{Gray}
	\put{\beginpicture
    \setcoordinatesystem units <0.7cm,0.7cm>         
    \setplotarea x from -0.7 to 0.7, y from -0.7 to 0.7  
        \put{\lfoldalt} at 0 0
        \put{\rfold} at 0.1 0.17
        \put{\ua} at 0.1 0.51
    \endpicture} at 0 1.35
    \put{\beginpicture
    \setcoordinatesystem units <0.7cm,0.7cm>         
    \setplotarea x from -0.7 to 0.7, y from -0.7 to 0.7  
        \put{\lfoldalt} at 0 0
        \put{\ra} at 0.3 -0.12
        \put{\upra} at 0.85 -0.1
    \endpicture} at 1.2 0.9
    \put{\beginpicture
    \setcoordinatesystem units <0.7cm,0.7cm>         
    \setplotarea x from -0.7 to 0.7, y from -0.7 to 0.7  
        \put{\nla} at 0 0
        \put{\dfold} at -0.35 -0.15
        \put{\upla} at -0.7 0
    \endpicture} at -1.2 0.9
    \put{\beginpicture
    \setcoordinatesystem units <0.7cm,0.7cm>         
    \setplotarea x from -0.7 to 0.7, y from -0.7 to 0.7  
        \put{\nla} at 0 0
        \put{\da} at -0.25 -0.45
        \put{\nla} at -0.53 -0.9
    \endpicture} at -1.2 -0.5
    \put{\beginpicture
    \setcoordinatesystem units <0.7cm,0.7cm>         
    \setplotarea x from -0.7 to 0.7, y from -0.7 to 0.7  
        \put{\nla} at 0 0
        \put{\da} at -0.25 -0.45
        \put{\lfold} at -0.25 -0.75
    \endpicture} at 0 0.2
    \color{black}
\endpicture
$$
$$
\vec u=s_1s_2s_1:\beginpicture
    \setcoordinatesystem units <1.2cm,1.2cm>         
    \setplotarea x from -2 to 2, y from -1.75 to 1.75  
    \put{\beginpicture
    \setcoordinatesystem units <0.7cm,0.7cm>         
    \setplotarea x from -1 to 1, y from -1 to 1  
        \put{\ra} at 0 0
        \put{\da} at 0.25 -0.45
        \put{\nla} at 0 -0.9
        \put{\da} at -0.25 -1.35
    \endpicture} at 0 -1.65
    \put{\beginpicture
    \setcoordinatesystem units <0.7cm,0.7cm>         
    \setplotarea x from -1 to 1, y from -1 to 1  
        \put{\rfold} at 0 0
        \put{\nla} at -0.28 -0.1
        \put{\da} at -0.55 -0.55
        \put{\nla} at -0.8 -1
    \endpicture} at -2 -1
    \put{\beginpicture
    \setcoordinatesystem units <0.7cm,0.7cm>         
    \setplotarea x from -1 to 1, y from -1 to 1  
        \put{\ra} at 0 0
        \put{\da} at 0.25 -0.45
        \put{\lfold} at 0.23 -0.78
        \put{\ra} at 0.53 -1.04
    \endpicture} at 2 -1.05
    \put{\beginpicture
    \setcoordinatesystem units <0.7cm,0.7cm>         
    \setplotarea x from -1 to 1, y from -1 to 1  
        \put{\rfold} at 0 0
        \put{\nla} at -0.28 -0.1
        \put{\da} at -0.55 -0.55
        \put{\lfold} at -0.6 -0.85
    \endpicture} at -1 0
    \color{Gray}
    \put{\beginpicture
    \setcoordinatesystem units <0.7cm,0.7cm>         
    \setplotarea x from -1 to 1, y from -1 to 1  
        \put{\ra} at 0 0
        \put{\da} at 0.25 -0.45
        \put{\nla} at 0 -0.9
        \put{\dfold} at -0.34 -1.05
    \endpicture} at 0 -0.05
    \color{black}
    \put{\beginpicture
    \setcoordinatesystem units <0.7cm,0.7cm>         
    \setplotarea x from -1 to 1, y from -1 to 1  
        \put{\ra} at 0 0
        \put{\da} at 0.25 -0.45
        \put{\lfold} at 0.23 -0.78
        \put{\rfold} at 0.4 -0.8
    \endpicture} at 1 0
    \put{\beginpicture
    \setcoordinatesystem units <0.7cm,0.7cm>         
    \setplotarea x from -1 to 1, y from -1 to 1  
        \put{\rfold} at 0 0
        \put{\nla} at -0.28 -0.1
        \put{\dfold} at -0.6 -0.2
        \put{\upla} at -0.95 -0.05
    \endpicture} at -2 1
    \put{\beginpicture
    \setcoordinatesystem units <0.7cm,0.7cm>         
    \setplotarea x from -1 to 1, y from -1 to 1  
        \put{\rfold} at 0 0
        \put{\lfoldalt} at -0.09 0.17
        \put{\ra} at 0.21 0.09
        \put{\upra} at 0.72 0.1
    \endpicture} at 2 1
    \put{\beginpicture
    \setcoordinatesystem units <0.7cm,0.7cm>         
    \setplotarea x from -1 to 1, y from -1 to 1  
        \put{\rfold} at 0 0
        \put{\lfoldalt} at -0.09 0.17
        \put{\rfold} at 0.01 0.32
        \put{\ua} at 0 0.67
    \endpicture} at -0.55 1.75
    \color{Gray}
    \put{\beginpicture
    \setcoordinatesystem units <0.7cm,0.7cm>         
    \setplotarea x from -0.7 to 0.7, y from -0.7 to 0.7  
        \put{\ra} at 0 0.05
        \put{\beginpicture   
        		\setplotarea x from -0.7 to 0.7, y from -0.7 to 0.7
    		\arrow <5pt> [.2,.67] from 0.08 -0.23 to 0.08 0.23 %
    		\plot 0.08 -0.23 -0.08 -0.23 /
    		\plot -0.08 -0.23 -0.08 0.05 /
	\endpicture} at 0.37 -0.17
        \put{\upla} at 0.2 0.2
        \put{\ua} at -0.05 0.6
    \endpicture} at 0.45 1.75
    \color{black}
\endpicture
$$
For each $\mu$, the paths in $\cP(\vec\la)_{\mu}$ with maximal dimension are shown in black. These are the \textit{LS galleries} of \cite[Definition~18]{litt}. The antidominant $\vec\la$-path is the bottom path in the picture.
\end{Example}

The following Lemma is proved in \cite[(4.5)]{schwer} and \cite[\S5 Proposition~4]{litt}. We provide a proof in our language here.

\begin{lemma}\label{lem:bound} Let $p\in\cP(\vec\la)_{\mu}$. Then 
$$
\dim(p)\leq\langle\la+\mu,\rho\rangle
\quad\textrm{with equality if and only if\quad
$\ell(\varphi(p))=f(p)$ and $\ell(p)=\ell(t_{\la})$}.
$$
\end{lemma}

\begin{proof} Let $p\in\cP(\vec\la)_{\mu}$ be of type $\vec u\cdot \vec m_\lambda$. 
Following \cite[Remark~4.4]{ram2}, for each $0\le i\le f(p)$ let $p_i$
be the positively folded alcove walk that agrees with $p$ up to the $i$th
fold and is unfolded thereafter.  Then $p_{f(p)}=p$, $p_0 = \vec u\cdot \vec m_\lambda$,
$$\varphi(p_i) = s_\alpha\varphi(p_{i-1}),\quad
\ell(\varphi(p_i))>\ell(\varphi(p_{i-1})),
\quad\hbox{and}\quad
\ell(\varphi(p_i)) - \ell(\varphi(p_{i-1}))-1\in 2\ZZ_{\ge 0},$$
if the $i$th fold is on the hyperplane $H_{\alpha+k\delta}$.
By (\ref{eq:epsilformula}) and (\ref{eq:signedlength}), 
$$\ell(p) + \ell(\varphi(p_0)) = \ell(t_{\lambda}) = \langle \lambda, 2\rho\rangle\qquad\textrm{and}\qquad\epsilon(p)+\ell(\varphi(p))=\langle\mu,2\rho\rangle$$
and it follows,  using (\ref{eq:dimension}), that
\begin{align*}
2(\langle \lambda+\mu, \rho\rangle - \dim(p))
&=\ell(\varphi(p))-f(p)+\ell(\varphi(p_0)) \\
&=2 \ell(\varphi(p_0)) + (\ell(\varphi(p))-\ell(\varphi(p_0))-f(p))\\
&=2\ell(\varphi(p_0)) + \sum_{i=1}^{f(p)} (\ell(\varphi(p_i))-\ell(\varphi(p_{i-1})-1)  \geq 0,
\end{align*}
with $2\ell(\varphi(p_0))$ and $\ell(\varphi(p))-\ell(\varphi(p_0))-f(p)$ nonnegative even integers.
The second equality implies that 
$\dim(p)=\langle\la+\mu,\rho\rangle$ if and only if $\ell(\varphi(p))-\ell(\varphi(p_0))-f(p)=0$ and $\ell(\varphi(p_0))=0$. 
Therefore $\dim(p)=\langle\la+\mu,\rho\rangle$ if and only if 
$\ell(\varphi(p))=f(p)$ and $\ell(\varphi(p_0))=\ell(t_{\la})-\ell(p)=0$.
\end{proof}

\begin{remark}\label{rem:folds} If $p\in\cP(\vec\la)_{\mu}$ has type $\vec u\cdot\vec m_{\la}$ with $u\in W^{\la}_0$ then the condition $\ell(p)=\ell(t_{\la})$ implies that $um_{\la}=t_{w_0\la}$. Thus $\dim(p)=\langle\la+\mu,\rho\rangle$ if and only if $p_0=p_{w_0\la}$ (with $p_0$ the straightening of $p$, as in the proof of Lemma~\ref{lem:bound}) and $\ell(\varphi(p))=f(p)$. \end{remark}

\subsection{Root operators}

Let $1\leq i\leq n$. The \textit{raising root operator} $\tilde{e}_i$ (cf. \cite{litt}, \cite{ram2}) is an operator on positively folded alcove walks that either kills a walk (i.e. $\tilde e_i(p)=0$) or produces a new positively folded alcove walk. The definition of $\tilde{e}_i(p)$ is as follows.

\begin{defn}The \textit{$i$-critical hyperplane} of an alcove walk $p$ is the hyperplane $H_{\alpha_i+k\delta}$ with $k$ maximal subject to the condition that $p$ makes a negative crossing on $H_{\alpha_i+k\delta}$.  The \textit{$i$-critical crossing} of $p$ is the first negative crossing of $p$ occurring on the $i$-critical hyperplane. If $p$ has no $i$-critical crossing then $\tilde{e}_i(p)=0$. Otherwise, let 
$
H_{\gamma}=H_{\alpha_i+k\delta}$ be the $i$-critical hyperplane
and compute $\tilde{e}_i(p)$ according to the following cases. 
\begin{enumerate}
    \item[(a)] If $p$ has a fold on $H_{\gamma+\delta}$ after the $i$-critical crossing then
$$
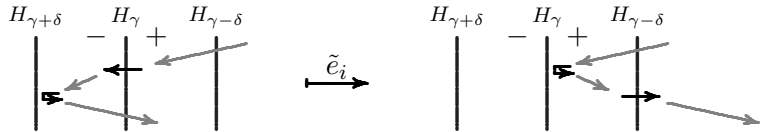
\beginpicture
    \setcoordinatesystem units <0.4cm,0.175cm>         
    \setplotarea x from -4 to 4, y from -4 to 5  
   \color{black}
    \put{\beginpicture
    \plot -3 -3.5 -3 3.5 /
    \plot 0 -3.5 0 3.5 /
    \plot 3 -3.5 3 3.5 /
    \color{Gray}
    \arrow <5pt> [.2,.67] from 4 3 to 1 1.5
    \color{black}
    \arrow <5pt> [.2,.67] from 0.5 1 to -0.7 1
    \color{Gray}
    \arrow <5pt> [.2,.67] from -1 0.5 to -2 -0.5
    \color{black}
    \plot -2.25 -0.75 -2.75 -0.75 -2.75 -1.25 /
    \arrow <5pt> [.2,.67] from -2.75 -1.25 to -2.15 -1.25
    \color{Gray}
    \arrow <5pt> [.2,.67] from -2 -1.5 to 1 -3
    \color{black}
    \endpicture} at -7 0
    \put{$\scriptstyle{H_{\gamma}}$} at -7 5
    \put{$\scriptstyle{H_{\gamma+\delta}}$} at -10 5
    \put{$\scriptstyle{H_{\gamma-\delta}}$} at -4 5
    \put{$-$} at -8 3.5
    \put{$+$} at -6 3.5
    \put{\beginpicture
    \plot -3 -3.5 -3 3.5 /
    \plot 0 -3.5 0 3.5 /
    \plot 3 -3.5 3 3.5 /
    \color{Gray}
    \arrow <5pt> [.2,.67] from 4 3 to 1 1.5
    \color{black}
    \plot 0.75 1.25 0.25 1.25 0.25 0.75 /
    \arrow <5pt> [.2,.67] from 0.25 0.75 to 0.85 0.75
    \color{Gray}
    \arrow <5pt> [.2,.67] from 1 0.5 to 2 -0.5
    \color{black}
    \arrow <5pt> [.2,.67] from 2.5 -1 to 3.7 -1
    \color{Gray}
    \arrow <5pt> [.2,.67] from 4 -1.5 to 7 -3
    \color{black}
    \endpicture} at 7 0
    \arrow <5pt> [.2,.67] from -1 0 to 1 0
    \put{$\tilde{e}_i$} at 0 1.25
    \plot -1 -0.25 -1 0.25 /
    \put{$\scriptstyle{H_{\gamma}}$} at 7 5
    \put{$\scriptstyle{H_{\gamma+\delta}}$} at 4 5
    \put{$\scriptstyle{H_{\gamma-\delta}}$} at 10 5
    \put{$-$} at 6 3.5
    \put{$+$} at 8 3.5
\endpicture
$$
    \item[(b)] If $p$ has a positive crossing on $H_{\gamma}$ after the $i$-critical crossing then
$$
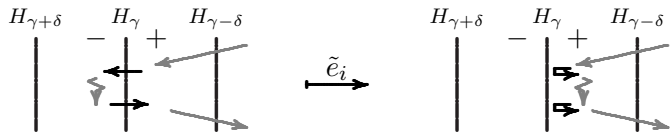
\beginpicture
    \setcoordinatesystem units <0.4cm,0.175cm>         
    \setplotarea x from -4 to 4, y from -4 to 5  
    \put{\beginpicture
    \plot -3 -3.5 -3 3.5 /
    \plot 0 -3.5 0 3.5 /
    \plot 3 -3.5 3 3.5 /
    \color{Gray}
    \arrow <5pt> [.2,.67] from 4 3 to 1 1.5
    \color{black}
    \arrow <5pt> [.2,.67] from 0.5 1 to -0.7 1
    \color{Gray}
    \plot -1 0.5 -1.25 0.25 -0.75 -0.25 -1 -0.5 /
    \arrow <5pt> [.2,.67] from -1 -0.5 to -1 -1.5
    \color{black}
    \arrow <5pt> [.2,.67] from -0.5 -1.5 to 0.7 -1.5
    \color{Gray}
    \arrow <5pt> [.2,.67] from 1.5 -2 to 4 -3.25
    \color{black}
    \endpicture} at -7 0
    \put{$\scriptstyle{H_{\gamma}}$} at -7 5
    \put{$\scriptstyle{H_{\gamma+\delta}}$} at -10 5
    \put{$\scriptstyle{H_{\gamma-\delta}}$} at -4 5
    \put{$-$} at -8 3.5
    \put{$+$} at -6 3.5
    \put{\beginpicture
    \plot -3 -3.5 -3 3.5 /
    \plot 0 -3.5 0 3.5 /
    \plot 3 -3.5 3 3.5 /
    \color{Gray}
    \arrow <5pt> [.2,.67] from 4 3 to 1 1.5
    \color{black}
    \plot 0.75 1.25 0.25 1.25 0.25 0.75 /
    \arrow <5pt> [.2,.67] from 0.25 0.75 to 1 0.75
    \color{Gray}
    \plot 1.2 0.5 1.45 0.25 0.95 -0.25 1.2 -0.5 /
    \arrow <5pt> [.2,.67] from 1.2 -0.5 to 1.2 -1.5
    \color{black}
    \plot 0.75 -1.5 0.25 -1.5 0.25 -2 /
    \arrow <5pt> [.2,.67] from 0.25 -2 to 1 -2
    \color{Gray}
    \arrow <5pt> [.2,.67] from 1.5 -2 to 4 -3.25
    \color{black}
    \endpicture} at 7 0
    \arrow <5pt> [.2,.67] from -1 0 to 1 0
    \put{$\tilde{e}_i$} at 0 1.25
    \plot -1 -0.25 -1 0.25 /
    \put{$\scriptstyle{H_{\gamma}}$} at 7 5
    \put{$\scriptstyle{H_{\gamma+\delta}}$} at 4 5
    \put{$\scriptstyle{H_{\gamma-\delta}}$} at 10 5
    \put{$-$} at 6 3.5
    \put{$+$} at 8 3.5
\endpicture
$$
    \item[(c)] The case when $p$ has neither a fold on $H_{\gamma+\delta}$ nor a positive crossing on $H_{\gamma}$ is computed as a degenerate case of case (a) (or (b)). It occurs when $p$ has no interactions with $\alpha_i+\ZZ\delta$ hyperplanes after the $i$-critical crossing on $H_{\gamma}$.
\end{enumerate}
\end{defn}    

\begin{remark} The above definition of $\tilde{e}_i$ is a synthesis of the operators $e_{\alpha}$ and $\tilde{e}_{\alpha}$ (with $\alpha=\alpha_i$) from \cite[Definitions~19 and~21]{litt}. 
\end{remark}

\begin{prop}\label{prop:rt} If $p$ is positively folded and $\tilde{e}_i(p)\neq0$ then $\tilde{e}_i(p)$ is positively folded and $\dim(\tilde{e}_i(p))=\dim(p)+1$. In case (a) $\mathrm{end}(\tilde{e}_i(p))=t_{\alpha^{\vee}_i}\mathrm{end}(p)$, where $t_{\alpha^{\vee}_i}$ is the translation in $\alpha_i^{\vee}$, in case (b) $\mathrm{end}(\tilde{e}_i(p))=\mathrm{end}(p)$, and in case (c) $\mathrm{end}(\tilde{e}_i(p))=s_{\gamma}\mathrm{end}(p)$.
\end{prop}

\begin{proof} Recall that $\dim(p) = \epsilon^+(p)+f(p)$ and consider case (a). 
The first grey part of $p$ is unchanged by $\tilde{e}_i$, and 
the crossing on $H_{\gamma}$ becomes a fold, which increases dimension by~$1$. 
The middle grey part of $p$ is reflected in $H_{\gamma}$ and, 
since $s_i$ permutes $R^+\backslash\{\alpha_i\}$, 
this reflected part remains positively folded and positive crossings in this portion remain positive after 
the reflection.  Thus the dimension contribution of this part is unchanged. 
The fold on $H_{\gamma+\delta}$ becomes a positive crossing, 
which does not change dimension, and 
the final grey part of $p$ is translated by $\alpha_i^{\vee}$, 
since $s_{\gamma}s_{\gamma+\delta}=t_{\alpha_i^{\vee}}$. 
Therefore $\tilde{e}_i(p)$ is positively folded with dimension $\dim(p)+1$ 
and $\mathrm{end}(\tilde{e}_i(p))=t_{\alpha_i^{\vee}}\mathrm{end}(p)$. 
Cases (b) and (c) are similar.
\end{proof}

\begin{cor}\label{cor:push} Let $p\in\cP(\vec\la)_{\mu}$. If $\dim(p)=\langle\la+\mu,\rho\rangle$ and $\tilde{e}_i(p)\neq 0$, then $\tilde{e}_i(p)\in\cP(\vec\la)_{\mu+\alpha_i^{\vee}}$.
\end{cor}

\begin{proof} If $p\in\cP(\vec\la)_{\mu}$ and $\tilde{e}_i(p)\neq 0$ then $\tilde{e}_i(p)\in\cP(\vec\la)_{\mu+\alpha_i^{\vee}}$ or $\tilde{e}_i(p)\in\cP(\vec\la)_{\mu}$ (with the former occurring in case (a) and the latter in case (b), with either possible in case (c)). By
Proposition~\ref{prop:rt} $\dim(\tilde e_i(p)) = \langle \lambda+\mu,\rho\rangle + 1$, and thus,
by Lemma~\ref{lem:bound}, $\tilde{e}_i(p)\notin\cP(\vec\la)_{\mu}$.
\end{proof}

\begin{Example}\label{ex:2} Let $p$ be the alcove walk from Example~\ref{ex:1}. The $1$-critical hyperplane of $p$ is $H_{\alpha_1+5\delta}$, and the $1$-critical crossing is the 13th step of $p$. There is no $2$-critical  crossing. Therefore $\tilde{e}_2(p)=0$, and $\tilde{e}_1(p)$ is the walk
$$\beginpicture
\setcoordinatesystem units <0.8cm,0.8cm>         
\setplotarea x from -4 to 4, y from -3.2 to 3.2  
    \color{black}
    \plot 4.3 -2.944 5.2 -1.3856 /
    \plot 3.3 -2.944 5.2 0.346 /
    \plot 2.3 -2.944 5.2 2.078 /
    \plot 1.3 -2.944 4.7 2.944 /
    \plot 0.3 -2.944 3.7 2.944 /
    \plot -0.7 -2.944 2.7 2.944 /
    \plot -1.7 -2.944 1.7 2.944 /
    \plot -2.7 -2.944 0.7 2.944 /
    \plot -3.7 -2.944 -0.3 2.944 /
    \plot -3.7 -1.212 -1.3 2.944 /
    \plot -3.7 0.520 -2.3 2.944 /
    \plot -3.7 2.2516 -3.3 2.944 /
    \plot 4.3 2.944 5.2 1.3856 /
    \plot 3.3 2.944 5.2 -0.346 /
    \plot 2.3 2.944 5.2 -2.078 /
    \plot 1.3 2.944 4.7 -2.944 /
    \plot 0.3 2.944 3.7 -2.944 /
    \plot -0.7 2.944 2.7 -2.944 /
    \plot -1.7 2.944 1.7 -2.944 /
    \plot -2.7 2.944 0.7 -2.944 /
    \plot -3.7 2.944 -0.3 -2.944 /
    \plot -3.7 1.212 -1.3 -2.944 /
    \plot -3.7 -0.520 -2.3 -2.944 /
    \plot -3.7 -2.2516 -3.3 -2.944 /
    \plot -3.7 -2.598 5.2 -2.598 /
    \plot -3.7 -1.732 5.2 -1.732 /
    \plot -3.7 -0.866 5.2 -0.866 /
    \plot -3.7 0 5.2 0 /
    \plot -3.7 2.598 5.2 2.598 /
    \plot -3.7 1.732 5.2 1.732 /
    \plot -3.7 0.866 5.2 0.866 /
    \arrow <5pt> [.2,.67] from -2 0.577 to -1.5 0.289   %
    \arrow <5pt> [.2,.67] from -1.5 0.289 to -1.5 -0.289   %
    \arrow <5pt> [.2,.67] from -1.5 -0.289 to -1 -0.577   %
    \arrow <5pt> [.2,.67] from -1 -0.577 to -1 -1.082   %
    \put{\beginpicture
    \setcoordinatesystem units <0.75cm,0.75cm>         
    \plot 0 0 -0.28 -0.15 /
    \plot -0.28 -0.15 -0.2 -0.3 /
    \arrow <5pt> [.2,.67] from -0.2 -0.3 to 0.05 -0.13
    \endpicture} at -1 -1.082
    \arrow <5pt> [.2,.67] from -0.95 -1.155 to -0.55 -1.432   %
    \put{\beginpicture
    \setcoordinatesystem units <0.7cm,0.7cm>         
    \arrow <5pt> [.2,.67] from 0.08 -0.32 to 0.08 0 %
    \plot 0.07 -0.32 -0.08 -0.32 /
    \plot -0.08 -0.32 -0.08 0 /
    \endpicture} at -0.5 -1.417
    \arrow <5pt> [.2,.67] from -0.45 -1.432 to 0 -1.155   %
    \arrow <5pt> [.2,.67] from 0 -1.155 to 0.5 -1.443   %
    \arrow <5pt> [.2,.67] from 0.5 -1.443 to 1 -1.155   %
    \arrow <5pt> [.2,.67] from 1 -1.155 to 1.5 -1.443   %
    \arrow <5pt> [.2,.67] from 1.5 -1.443 to 1.97 -1.23   %
    \put{\beginpicture
    \setcoordinatesystem units <0.7cm,0.7cm>         
    \arrow <5pt> [.2,.67] from 0.28 -0.15 to 0 0
    \plot 0.28 -0.15 0.2 -0.3 /
    \plot 0.2 -0.3 -0.05 -0.13 /
    \endpicture} at 2 -1.155
    \arrow <5pt> [.2,.67] from 2 -1.13 to 2 -0.577   %
    \arrow <5pt> [.2,.67] from 2 -0.577 to 1.5 -0.289   %
    \arrow <5pt> [.2,.67] from 1.5 -0.289 to 1.5 0.289   %
    \arrow <5pt> [.2,.67] from 1.5 .289 to 1 0.577   %
    \arrow <5pt> [.2,.67] from 1 0.577 to 1 1.155   %
    \arrow <5pt> [.2,.67] from 1 1.155 to 0.5 1.443   %
    \arrow <5pt> [.2,.67] from 0.5 1.443 to 0.5 2.021   %
    \color{Gray}
    \arrow <5pt> [.2,.67] from 2.27 -1.35 to 2.5 -1.443   %
    \arrow <5pt> [.2,.67] from 2.5 -1.443 to 2.97 -1.23   %
    \put{\beginpicture
    \setcoordinatesystem units <0.7cm,0.7cm>         
    \arrow <5pt> [.2,.67] from 0.28 -0.15 to 0 0
    \plot 0.28 -0.15 0.2 -0.3 /
    \plot 0.2 -0.3 -0.05 -0.13 /
    \endpicture} at 3 -1.155
    \arrow <5pt> [.2,.67] from 3 -1.13 to 3 -0.577   %
    \arrow <5pt> [.2,.67] from 3 -0.577 to 2.5 -0.289   %
    \arrow <5pt> [.2,.67] from 2.5 -0.289 to 2.5 0.289   %
    \arrow <5pt> [.2,.67] from 2.5 0.289 to 2 0.577   %
    \arrow <5pt> [.2,.67] from 2 0.577 to 2 1.155   %
    \color{black}
\endpicture$$
Note that the end vertex is translated by $\alpha_1^{\vee}$. The operator $\tilde{e}_1$ can be applied~$5$ more times before killing the walk.
\end{Example}

\section{Proof of Theorem $\mathbf{C}'$}

A subset $\Pi\subseteq P$ is \textit{saturated} if for all $\alpha\in R$ and $\mu\in\Pi$, if $k$ is between $0$ and $\langle\mu,\alpha\rangle$ (inclusive) then $\mu-k\alpha^{\vee}\in\Pi$ (in particular, $W_0\Pi=\Pi$). For each $\la\in P^+$ there is a unique saturated set $\Pi_{\la}$ with the property that $\mu\preceq\la$ for all $\mu\in\Pi_{\la}$, where the partial order is given by $\mu\preceq\la$ if and only if $\la-\mu\in Q^+$. Explicitly we have (see \cite[(2.6.2)]{m} for example)
$$
\Pi_{\la}=\{w\mu\mid\mu\preceq\la,\mu\in P^+,w\in W_0\}=\bigcap_{w\in W_0}w(\la- Q^+)=\mathrm{conv}(W_0\la)\cap (\la+Q).
$$
The following lemma is modelled on the \textit{adapted strings} of \cite{littelmanncrystals} (see also \cite[Proposition~6.5]{kamnitzer}).

\begin{lemma}\label{lem:fall} Let $\la\in P^+$ and $\mu\in \Pi_{\la}$. Fix a reduced decomposition $w_0=s_{i_1}\cdots s_{i_N}$ of the longest element of $W_0$ and define coweights $\mu_0,\ldots,\mu_N\in\Pi_{\la}$ by $\mu_0=\mu$ and
$$
\mu_k=\mu_{k-1}-m_k\alpha_{i_k}^{\vee},\qquad\textrm{where}\qquad m_k=\max\{m\in\ZZ_{\geq0}\mid \mu_{k-1}-m\alpha_{i_k}^{\vee}\in\Pi_{\la}\}
$$
for $1\leq k\leq N$. Then $\mu_{N}=w_0\la$.
\end{lemma}

\begin{proof} Define coweights $\la_0,\ldots,\la_N\in\Pi_{\la}$ by $\la_0=\la$ and
$$
\la_{k}=s_{i_k}\la_{k-1}=\la_{k-1}-\langle\la_{k-1},\alpha_{i_k}\rangle\alpha_{i_k}^{\vee}\quad\textrm{for $1\leq k\leq N$}.
$$
These are vertices of the polytope $\mathrm{Conv}(W_0\lambda)$, 
the convex hull of the $W_0$-orbit of $\lambda$.  See the picture in Example \ref{museq}.
We will show that $\mu_k\preceq\la_k$ for all $k=0,\ldots,N$. The result of Lemma~\ref{lem:fall} will follow, because $\mu_N\preceq\la_N=w_0\la$ and $\mu_N\in\Pi_{\la}$ imply that $\mu_N=w_0\la$.

We have $\mu_0=\mu\preceq\la=\la_0$, and by induction $\la_{k-1}-\mu_{k-1}\in\ZZ_{\geq0}\alpha_1^{\vee}+\cdots+\ZZ_{\geq0}\alpha_n^{\vee}$, and so
\begin{align*}
\la_k-\mu_k=\la_{k-1}-\mu_{k-1}+(m_k-\langle\la_{k-1},\alpha_{i_k}\rangle)\alpha_{i_k}^{\vee}=\gamma+(m_k+c-\langle\la_{k-1},\alpha_{i_k}\rangle)\alpha_{i_k}^{\vee}
\end{align*}
where $\lambda_{k-1}-\mu_{k-1} = \gamma+c\alpha_{i_k}^\vee$ with 
$\gamma\in\sum_{i\neq i_k}\ZZ_{\geq0}\alpha_i^{\vee}$ and 
$c\alpha_{i_k}^{\vee}\in\ZZ_{\geq0}\alpha_{i_k}^{\vee}$. 
We will show that $m_k+c-\langle\la_{k-1},\alpha_{i_k}\rangle\geq0$. 

We have
$$
0\leq\langle\la_{k-1},\alpha_{i_k}\rangle=\langle\mu_{k-1}+c\alpha_{i_k}^{\vee}+\gamma,\alpha_{i_k}\rangle\leq\langle\mu_{k-1}+c\alpha_{i_k}^{\vee},\alpha_{i_k}\rangle
$$
(since $\langle\la_{k-1},\alpha_{i_k}\rangle=\langle\la,s_{i_1}\cdots s_{i_{k-1}}\alpha_{i_k}\rangle\geq0$ and $\langle\alpha_i^{\vee},\alpha_j\rangle\leq0$ for $i\neq j$). Thus, by the property of saturated sets, we have
$$
\mu_{k-1}-(\langle\la_{k-1},\alpha_{i_k}\rangle-c)\alpha_{i_k}^{\vee}=(\mu_{k-1}+c\alpha_{i_k}^{\vee})-\langle\la_{k-1},\alpha_{i_k}\rangle\alpha_{i_k}^{\vee}\in\Pi_{\la},
$$
and so by the definition of $m_k$ we have $m_k\geq\langle\la_{k-1},\alpha_{i_k}\rangle-c$. 
Thus $m_k+c-\langle \lambda_{k-1},\alpha_{i_k}^\vee\rangle\ge 0$ which completes the
induction step.  
\end{proof}

\begin{proof}[Proof of Theorem $\mathbf{C}'$] Suppose that $p\in\cP(\vec\la)_{\mu}$. If $p$ has type $\vec w$ (with $w\in\tilde{W}$) then the end alcove $v\in\tilde{W}$ of $p$ is a subexpression of $\vec w$, and thus $v\leq w$ in Bruhat order. Since $w(0)\in\Pi_{\la}$, using \cite[Proposition~3.1]{p2} we have $v(0)\in\Pi_{\la}$. But $v(0)=\mu$. Thus if $\mu\notin\Pi_{\la}$ then $\cP(\vec\la)_{\mu}=\emptyset$. 

Suppose that $\mu\in \Pi_{\la}$. We will construct a path $p\in\cP(\vec\la)_{\mu}$ with $\dim(p)=\langle\la+\mu,\rho\rangle$. Let $w_0=s_{i_1}\cdots s_{i_N}$ be a reduced expression, and define $\mu_k$ and $m_k$ as in Lemma~\ref{lem:fall}. Define $p_0,\ldots,p_N$ by $p_N=p_{w_0\la}$ (the canonical antidominant path) and 
$$
p_{k-1}=\tilde{e}_{i_{k}}^{m_{k}}(p_{k}),\qquad\textrm{for $k=1,\ldots,N$.}
$$
We claim that $p_k\neq 0$ for all $k=0,\ldots,N$, and that
\begin{align*}
p_k\in\cP(\vec\la)_{\mu_k}\quad\textrm{and}\quad\dim(p_k)=\langle\la+\mu_k,\rho\rangle.
\end{align*}
By Lemma~\ref{lem:fall},  $p_N\in\cP(\vec\la)_{\mu_N}$ and $\dim(p_N)=0=\langle\la+w_0\la,\rho\rangle$, starting a descending induction. Let $\beta_k=\alpha_{i_k}+(m_k-\langle\mu_{k-1},\alpha_{i_k}\rangle)\delta$.  With the orientation from Section~\ref{subsect:alcovewalk}, we have
\begin{align*}
\begin{aligned}
0&\in H_{\beta_k}^+,\\
\mu_{k}&\in H_{\beta_k+m_k\delta},
\end{aligned}\quad
\begin{aligned}
&\textrm{because}\\
&\textrm{because}
\end{aligned}\quad
\begin{aligned}
\langle 0,\beta_k\rangle&=m_k-\langle\mu_{k-1},\alpha_{i_k}\rangle\geq0,\quad\textrm{and}\\
\langle\mu_k,\beta_k\rangle&=0.
\end{aligned}
\end{align*}
(We have $m_k\geq\langle\mu_{k-1},\alpha_{i_k}\rangle$ because $s_{i_k}\mu_{k-1}=\mu_{k-1}-\langle\mu_{k-1},\alpha_{i_k}\rangle\alpha_{i_k}^{\vee}\in\Pi_{\la}$). Therefore any alcove walk $p$ starting at $1\in\tilde{W}$ with end vertex $\mu_k$ must make negative crossings on each of the hyperplanes $H_{\beta_k},H_{\beta_k+\delta},\ldots,H_{\beta_k+(m_k-1)\delta}$, as illustrated:
\begin{align}\label{eq:pict}
\beginpicture
\setcoordinatesystem units <0.8cm,0.8cm>         
\setplotarea x from -0.8 to 0.7, y from -0.5 to 0.5  
\put{$\scriptstyle{H_{\beta_k+(m_k-1)\delta}}$}[b] at 0 0.6
\put{$\scriptstyle{-}$}[b] at -0.4 0.25
\put{$\scriptstyle{+}$}[b] at 0.4 0.25
\put{$\scriptstyle{\mu_k}$}[br] at -0.45 -0.1
\plot  0 -0.4  0 0.5 /
\arrow <5pt> [.2,.67] from 0.3 0 to -0.35 0   %
\endpicture
\beginpicture
\setcoordinatesystem units <0.8cm,0.8cm>         
\setplotarea x from -0.8 to 0.7, y from -0.5 to 0.5  
\put{$\scriptstyle{-}$}[b] at -0.4 0.25
\put{$\scriptstyle{+}$}[b] at 0.4 0.25
\plot  0 -0.4  0 0.5 /
\arrow <5pt> [.2,.67] from 0.3 0 to -0.36 0   %
\endpicture\cdots
\beginpicture
\setcoordinatesystem units <0.8cm,0.8cm>         
\setplotarea x from -0.8 to 0.7, y from -0.5 to 0.5  
\put{$\scriptstyle{H_{\beta_k}}$}[b] at 0 0.6
\put{$\scriptstyle{-}$}[b] at -0.4 0.25
\put{$\scriptstyle{+}$}[b] at 0.4 0.25
\put{$\scriptstyle{1}$}[bl] at 0.45 -0.1
\plot  0 -0.4  0 0.5 /
\arrow <5pt> [.2,.67] from 0.3 0 to -0.35 0   %
\endpicture.
\end{align}
Thus by the definition of $\tilde{e}_{i_k}$ we have $\tilde{e}_{i_k}^{m_k}(p)\neq 0$ and so, by induction, $p_{k-1}=\tilde{e}_{i_k}^{m_k}(p_k)\neq 0$. Furthermore, by Corollary~\ref{cor:push} and induction, 
$p_{k-1}\in\cP(\vec\la)_{\mu_k+m_k\alpha_{i_k}^{\vee}}=\cP(\vec\la)_{\mu_{k-1}}$ and $\dim(p_{k-1})=\dim(p_k)+m_k=\langle\la+\mu_{k-1},\rho\rangle$. So $p_0\in\cP(\vec\la)_{\mu}$ and $\dim(p_0)=\langle\la+\mu,\rho\rangle$.
\end{proof}

Theorem $\mathbf{C}'$ has some interesting implications.  In the setup of Theorem $\mathbf{S}$ we see that the constant term of $c_{\la\mu}(q^{-1})$ is a positive integer. In the context of Theorem~$\mathbf{B}$, if the building $X$ has finite thickness $q$ (see Appendix~\ref{app:buildings}), then Theorem~$\mathbf{C'}$ gives
bounds on the number of points that chamber retract to $\lambda$ and sector retract to $\mu$. In the setup of Theorem~$\mathbf{G}$, with $k=\CC$, the \emph{affine Grassmannian} is $G/K$ and the
\emph{MV-cycles} of type $\lambda$ and weight $\mu$ are the irreducible components of 
$$
\overline{U^-t_\mu K \cap K t_\lambda K}
\qquad \textrm{(a subvariety of $G/K$)}.
$$
Gaussent and Littelmann \cite[\S11, Corollary~5]{litt} showed that there is a bijection between 
the MV-cycles of type $\lambda$ and weight $\mu$ and the elements of $\cP(\vec\lambda)_\mu$ 
with dimension $\langle \lambda+\mu,\rho\rangle$.  The work of Kamnitzer \cite[Theorem~C]{kamnitzer2} 
established a bijection between the MV-cycles and the MV-polytopes which appeared in the work 
of Anderson~\cite{andersonpoly}.  Kamnitzer \cite[Theorem~3.5]{kamnitzer} describes the MV-polytope corresponding to an MV-cycle in terms of its \emph{string parameters}
and part (c) of the following Corollary determines the string parameters for the MV-polytope
determined by the path $p$ constructed in Theorem~$\mathbf{C}'$.

\begin{cor}\label{cor:implications} Let $\la\in P^+$ and $\mu\in\Pi_{\la}$. 
\begin{enumerate}
\item[(a)] In Theorem~$\mathbf{S}$, the constant term in the polynomial $c_{\la\mu}(q^{-1})$ is a positive integer, and if $q\in \RR$ and $q>1$ then $c_{\la\mu}(q^{-1})\geq(1-q^{-1})^{\ell(w_0)}$.
\item[(b)] In Theorem~$\mathbf{B}$, if $X$ has finite thickness $q$ (see Appendix~\ref{app:buildings}) then
$$
|\{x\in X\mid\rho_{1}(x)\in W_0\la\textrm{ and }\rho_{-\infty}(x)=\mu\}|\geq (q-1)^{\langle\la+\mu,\rho\rangle}.
$$ 
\item[(c)] In Theorem~$\mathbf{G}$ with $k=\CC$, there is an irreducible component of $\overline{U^-t_{\mu}K\cap Kt_{\la}K}$ with dimension $\langle\la+\mu,\rho\rangle$. This $MV$-cycle has string parameters $(m_1,\ldots,m_N)$, where $m_k$ is defined in Lemma~\ref{lem:fall}.
\end{enumerate}
\end{cor}

\begin{proof} Let $p\in\cP(\vec\la)_{\mu}$ be the path constructed in Theorem~$\mathbf{C}'$. 

(a) The formula~(\ref{eq:macpath}) and the fact that~$p$ has at most $\ell(w_0)$ folds (the proof of Lemma~\ref{lem:bound} gives $f(p)\leq\ell(\varphi(p))$) proves~(a).

(b) The number of galleries with types $\{\vec u\cdot \vec m_{\la}\mid u\in W_0^{\la}\}$ which retract to $p$ under $\rho_{-\infty}$ is $q^{\epsilon^+(p)}(q-1)^{f(p)}\geq(q-1)^{\dim(p)}=(q-1)^{\langle\la+\mu,\rho\rangle}$.

(c) Using the labelling of points in $G/I$ from \cite[Theorem~7.1]{ramparkinsonschwer} 
we can write down a ``cell'' in $U^-t_{\mu}K\cap Kt_{\la}K$ isomorphic to $\CC^{\epsilon^+(p)}\times(\CC^{\times})^{f(p)}$. This gives a dense subset of a $\langle\la+\mu,\rho\rangle$ dimensional irreducible component $z$ of $\overline{U^-t_{\mu}K\cap Kt_{\la}K}$.  The
irreducible component $z$ is the MV-cycle corresponding to $p$ determined by 
\cite[\S11, Theorem~4]{litt}. 
The vertices $\mu_0,\ldots,\mu_N$ from the proof of Theorem~$\mathbf{C}'$ are a subset of the 
vertices of the corresponding MV-polytope and these vertices uniquely determine the polytope. 
The construction $p=\tilde{e}_{i_1}^{m_1}\cdots \tilde{e}_{i_N}^{m_N}(p_{w_0\la})$ in 
the proof of Theorem~$\mathbf{C}'$ can be translated to MV-polytopes by 
using the crystal structure on MV-polytopes from \cite[Theorem~3.5]{kamnitzer} and
applying the root operators to the antidominant polytope.  It follows that 
the string parameters of $z$ are $(m_1,\ldots,m_N)$ (relative to the reduced decomposition $w_0=s_{i_1}\cdots s_{i_{N}}$).
\end{proof}

\begin{Example} \label{museq}
Let us briefly illustrate how the lemma and theorem work in an example. Consider $\la$ and $\mu=\mu_0$ as shown below, and choose the reduced decomposition $w_0=s_1s_2s_1$.
$$
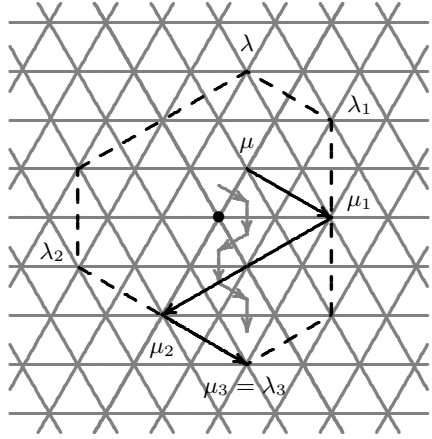
\beginpicture
\setcoordinatesystem units <0.75cm,0.75cm>         
\setplotarea x from -5 to 5, y from -4 to 4  
    \color{Gray}
    \plot 2.8 -3.81 3.7 -2.2516 /
    \plot 1.8 -3.81 3.7 -0.52 /
    \plot 0.8 -3.81 3.7 1.212 /
    \plot -0.2 -3.81 3.7 3.044 /
    \plot -1.2 -3.81 3.2 3.81 /
    \plot -2.2 -3.81 2.2 3.81 /
    \plot -3.2 -3.81 1.2 3.81 /
    \plot -3.7 -3.044 0.2 3.81 /
    \plot -3.7 -1.212 -0.8 3.81 /
    \plot -3.7 0.520 -1.8 3.81 /
    \plot -3.7 2.2516 -2.8 3.81 /
    \plot -2.8 -3.81 -3.7 -2.2516 /
    \plot -1.8 -3.81 -3.7 -0.52 /
    \plot -0.8 -3.81 -3.7 1.212 /
    \plot 0.2 -3.81 -3.7 3.044 /
    \plot 1.2 -3.81 -3.2 3.81 /
    \plot 2.2 -3.81 -2.2 3.81 /
    \plot 3.2 -3.81 -1.2 3.81 /
    \plot 3.7 -3.044 -0.2 3.81 /
    \plot 3.7 -1.212 0.8 3.81 /
    \plot 3.7 0.520 1.8 3.81 /
    \plot 3.7 2.2516 2.8 3.81 /
    \plot -3.7 -3.464 3.7 -3.464 /
    \plot -3.7 -2.598 3.7 -2.598 /
    \plot -3.7 -1.732 3.7 -1.732 /
    \plot -3.7 -0.866 3.7 -0.866 /
    \plot -3.7 0 3.7 0 /
    \plot -3.7 3.464 3.7 3.464 /
    \plot -3.7 2.598 3.7 2.598 /
    \plot -3.7 1.732 3.7 1.732 /
    \plot -3.7 0.866 3.7 0.866 /
    \plot -2.2 -3.81 2.2 3.81 /
    \plot 2.2 -3.81 -2.2 3.81 /
    \plot -3.7 0 3.7 0 /
    \color{black}
    \put{$\bullet$} at 0 0
    \put{{\footnotesize$\la$}} at 0.5 3.1
    \put{{\footnotesize$\la_1$}} at 2.5 2
    \put{{\footnotesize$\la_2$}} at -2.95 -0.6
    \put{{\footnotesize$\mu_3=\la_3$}} at 0.47 -3
    \put{{\footnotesize$\mu_2$}} at -1 -2.35
    \put{{\footnotesize$\mu_1$}} at 2.5 0.26
    \put{{\footnotesize$\mu$}} at 0.5 1.3
    \setdashes
    \plot -2.5 0.866 0.5 2.598 2 1.722 2 0 2 -1.722 0.5 -2.598 -2.5 -0.866 -2.5 0.866 / %
    \setsolid
    \color{Gray}
    \arrow <5pt> [.2,.67] from 0 0.577 to 0.5 0.289 %
    \arrow <5pt> [.2,.67] from 0.5 0.289 to 0.5 -0.289 %
    \arrow <5pt> [.2,.67] from 0.5 -0.289 to 0 -0.577 %
    \arrow <5pt> [.2,.67] from 0 -0.577 to 0 -1.155 %
    \arrow <5pt> [.2,.67] from 0 -1.155 to 0.5 -1.443 %
    \arrow <5pt> [.2,.67] from 0.5 -1.443 to 0.5 -2.021 %
    \color{black}
    \arrow <5pt> [.2,.67] from -1 -1.722 to 0.5 -2.598
    \arrow <5pt> [.2,.67] from 2 0 to -1 -1.722
    \arrow <5pt> [.2,.67] from 0.5 0.866 to 2 0
    \endpicture$$
We have $m_1=1$, $m_2=2$, and $m_3=1$. The path constructed in Theorem~$\mathbf{C'}$ is obtained from $p_{w_0\la}$ by performing the root operator $\tilde{e}_1$ once, then the root operator $\tilde{e}_2$ twice, and then the root operator $\tilde{e}_1$ once more.

In this example, choosing the other reduced expression $w_0=s_2s_1s_2$ produces a different path ending at $\mu$. These two paths correspond to the two elements of weight $\mu$ in the crystal of the representation $V(\la)$.
\end{Example}


\begin{appendix}

\section{Alcove walks and buildings}\label{app:buildings}

In this appendix we show how the chamber and sector retractions on regular affine buildings are precisely described by the combinatorics of labelled versions of alcove walks (see Theorem~\ref{thm:retractions}).

\subsection{The Coxeter complex $\Sigma$}

Use the terminology of Section~\ref{subsect:geom} so that $R$ is a reduced irreducible
root system in an $n$ dimensional real vector space $\fh_\RR$, $W=Q\rtimes W_0$ is an affine Weyl group 
and $\tilde{W}=P\rtimes W_0$ is an extended affine Weyl group. The alcoves of the hyperplane arrangement from Section~\ref{subsect:geom} are open geometric simplices. The extreme points of alcoves are \textit{vertices}, and the hyperplane arrangement induces the structure of a simplicial complex $\Sigma=\Sigma(W)$ (the \textit{Coxeter complex} of $W$) with the maximal simplices being the alcoves. The hyperplanes of the hyperplane arrangement form the \textit{walls} of~$\Sigma$. The walls of $\Sigma$ are orientated as in Section~\ref{subsect:alcovewalk}. The $\tilde{W}$-translates of the Weyl chambers are the \textit{sectors} of~$\Sigma$ 
(thus there are infinitely many sectors because they need not be based at~$0$).
The \textit{rank} of $\Sigma$ is $\dim(\fh_\RR)+1=n+1$ (the rank of an arbitrary Coxeter complex is the number of vertices in a maximal simplex).

The vertices of the fundamental alcove $c_0$ of $\Sigma$ are 
$\{0\}\cup\{\omega_i/k_i\mid i=1,\ldots,n\}$, where 
$\theta=k_1\alpha_1+\cdots+k_n\alpha_n$ is the highest root. 
For example, in type $B_2$ (with $\fh_{\RR}=\RR^2$, $\alpha_1=e_1-e_2$, $\alpha_2=e_2$, $\theta=\alpha_1+2\alpha_2$, $\omega_1=e_1$, and $\omega_2=e_1+e_2$) one has the picture
$$
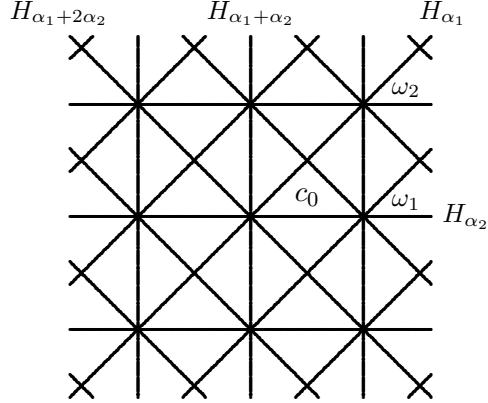
\beginpicture
\setcoordinatesystem units <0.75cm,0.75cm>         
\setplotarea x from -5 to 5, y from -3.5 to 4    
\put{\small{$H_{\alpha_1+\alpha_2}$}}[b] at 0 3.4
\put{\small{$H_{\alpha_1}$}}[b] at 3.4 3.4
\put{\small{$H_{\alpha_1+2\alpha_2}$}}[b] at -3.4 3.4
\put{\small{$H_{\alpha_2}$}}[l] at 3.4 0
\put{\small{$\omega_2$}}[l] at 2.5 2.25
\put{\small{$\omega_1$}}[l] at 2.5 0.25
\put{$c_0$}[b] at 1 0.2
\plot  0  3.2   0 -3.2 /
\plot  -2 3.2   -2 -3.2 /
\plot   2 3.2    2 -3.2 / %
\plot  3.2  0  -3.2  0 /
\plot  3.2 -2   -3.2 -2 /
\plot  3.2  2   -3.2  2 /  %
\plot 3.2 3.2 -3.2 -3.2 /
\plot  1.2 3.2   -3.2 -1.2 /
\plot  -0.8 3.2  -3.2 0.8 / 
\plot  3.2 -0.8   0.8 -3.2 /
\plot  3.2 1.2  -1.2 -3.2 / 
\plot  2.8 -3.2 3.2 -2.8 / 
\plot  -3.2 2.8 -2.8 3.2 / %
\plot  -3.2 3.2 3.2 -3.2 / 
\plot  -3.2 -0.8   -0.8 -3.2 /
\plot  -3.2 1.2   1.2 -3.2 /
\plot  -1.2 3.2   3.2 -1.2 /
\plot  0.8 3.2  3.2 0.8 / 
\plot  -3.2 -2.8 -2.8 -3.2 /
\plot  2.8 3.2 3.2 2.8 / %
\endpicture
$$

Define a 
\textit{type} function $\tau:\{\textrm{vertices of $\Sigma$}\}\to\{0,\ldots,n\}$ by setting $\tau(0)=0$,  $\tau(\omega_i/k_i)=i$ and
requiring each alcove to have exactly one vertex of each type (equivalently, $\tau(wx) = \tau(x)$ for each vertex $x$ and each $w\in W$). For $w\in W$ the alcoves $w$ and $ws_i$ have all vertices in common except for those of type~$i$.

\subsection{Buildings}\label{subsection:buildings}

A \textit{building of type $W$} is a nonempty simplicial complex $X$ with a collection $\cA$ of subcomplexes (the \textit{apartments} of $X$) such that
\begin{enumerate}
	\item[(B1)] If $A\in\cA$ then $A\cong \Sigma$, where $\Sigma$ is the Coxeter complex of $W$,
	\item[(B2)] If $c$ and $d$ are maximal simplices of $X$ then there is $A\in\cA$ such that $c\in A$ and $d\in A$, and
	\item[(B3)] If $A,A'\in\cA$ with $A\cap A'\neq\emptyset$ then there is an isomorphism $\psi:A'\to A$ with $\psi_{A\cap A'}=1$.
	\end{enumerate}
The maximal simplices of $X$ are \textit{chambers}; they correspond to the alcoves of $\Sigma$. By (B1) apartments have \textit{walls} and \textit{sectors}. A \textit{sector} of $X$ is a sector in some apartment of~$X$.  The \textit{rank} of $X$ is the rank of $\Sigma$.

We regard $\Sigma$ as an apartment of $X$ by fixing an apartment $A_0$ and an isomorphism $A_0\to\Sigma$. This induces a type function $\tau:\{\textrm{vertices of $X$}\}\to\{0,1,\ldots,n\}$ such that every chamber of $X$ has exactly one vertex of each type. A \textit{panel} of $X$ is a codimension~1 simplex. Thus a panel contains one vertex of all but one type. This type is the \textit{cotype} of the panel. Chambers $c$ and $d$ are \textit{$i$-adjacent} if $c$ and $d$ share a cotype~$i$ panel. 
$$
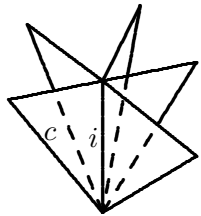
\beginpicture
\setcoordinatesystem units <0.25cm,0.25cm>         
\setplotarea x from -6 to 6, y from -2 to 9  
	\put{$c$} at -2.8 2.2
    \put{$i$} at -0.5 2
    \plot -5 4  0 -2 /
    \plot 0 5 -5 4 /
    \plot 0 -2 0 5 /
    \plot 0 -2  5 1 /
    \plot 5 1 0 5 /
    \plot 0 5  2 9 /
    \plot 2.917 2.666 5 6 /
    \plot  2 9  1.321 5.264 /
    \plot  0 5  5 6 /
    \plot -2.5926 4.4815  -4 8 /
    \plot  -4 8  0 5 /
    \setdashes
    \plot  0 -2  2.917 2.666 /
    \plot 1.321 5.264  0 -2 /
    \plot 0 -2  -2.5926 4.4815 /
\endpicture$$
A \textit{gallery} of type $\vec w=s_{i_1}\cdots s_{i_{\ell}}$ from $c$ to $d$ is a sequence $c=c_0,c_1,\ldots,c_{\ell}=d$ with $c_{k-1}$ $i_k$-adjacent to $c_k$ and $c_{k-1}\neq c_k$ for each $k=1,\ldots,\ell$. If this gallery has minimal length amongst the galleries from $c$ to $d$ then the element of $W$ given by $\delta(c,d)=s_{i_1}\cdots s_{i_{\ell}}$ does not depend on the particular minimal gallery from $c$ to $d$ chosen. Note that $\delta(c,d)=s_i$ if and only if $c\neq d$ and $c$ and $d$ are $i$-adjacent. A sequence $c=c_0,c_1,\ldots,c_{\ell}=d$ with $c_{k-1}$ $i_k$-adjacent to $c_k$ for each $k=1,\ldots,\ell$ (without imposing the condition $c_{k-1}\neq c_k$) is a \textit{pregallery} of type $\vec w=s_{i_1}\cdots s_{i_{\ell}}$ from $c$ to~$d$.

A building is \textit{thick} if every panel is contained in at least $3$ chambers,
and \textit{regular} if there is a bijection
\begin{align}\label{eq:reg}
\{\textrm{chambers containing $\pi$}\}\longleftrightarrow\{\textrm{chambers containing $\pi'$}\}
\end{align}
whenever $\pi$ and $\pi'$ are panels with the same cotype (we do not assume that this cardinality is finite). The \textit{thickness} of a regular building~$X$ is the $(n+1)$-tuple $(q_0,\ldots,q_n)$, where 
$$
q_i+1=\Card(\{\textrm{chambers containing $\pi_i$}\})
\qquad\textrm{for any cotype~$i$ panel~$\pi_i$.}
$$
In the case that $q_0=\cdots=q_n=q$ we say that $X$ has thickness~$q$.

The following proposition does not require the building to be affine, but as a consequence we see that thick irreducible affine buildings of rank at least $3$ are automatically regular. 

\begin{prop}\label{prop:bijection} Let $W$ be an irreducible Coxeter group and suppose that $s_is_j$ has finite order for all~$i$ and~$j$. Let $X$ be a thick building of type $W$. Then $X$ is regular. 
\end{prop}

\begin{proof} If $X$ has rank $2$ then $W$ is a dihedral group generated by $s_1s_2$, and by hypothesis $s_1s_2$ has finite order, $m$ say. Thus $X$ is a thick rank $2$ spherical building, and it follows from the geometry of \textit{generalised $m$-gons} that $X$ is regular
\cite[Proposition~3.3]{ronan}.

For general rank, let $\pi$ and $\pi'$ be panels with cotype~$i$, and let $c$ and $d$ be chambers of $X$ such that $\pi\subseteq c$ and $\pi'\subseteq d$. 
By induction on the length of $\delta(c,d)$ it suffices to consider the case when $\delta(c,d)=s_j$ for some~$j$. Therefore $c\cap d$ is a cotype~$j$ panel. If $j=i$ then $c\cap d=\pi=\pi'$, and the bijection (\ref{eq:reg}) is trivial. If $j\neq i$ then
$\{c'\mid \delta(c,c')\in \langle s_i,s_j\rangle\}$ is a rank 2 sub-building of spherical type $W_{\{i,j\}}=\langle s_i,s_j\rangle$ \cite[Theorem~3.5]{ronan}. This sub-building contains $c$ and $d$, and so the rank 2 case gives the bijection~(\ref{eq:reg}). 
\end{proof}

\subsection{Sectors and retractions}

There are two distinct types of retractions for affine buildings, 
chamber retractions and sector retractions. See \cite[\S IV.3 and \S VI.8]{brown}.
\begin{enumerate}
\item[$\bullet$]
Let $A$ be an apartment of $X$ and let $c$ be a chamber in~$A$.
The \textit{chamber retraction} $\rho_{A,c}:X\to A$ is defined as follows. Let $d$ be a chamber of  $X$. 
Choose an apartment $A'$ containing $d$ and $c$, and let $\psi:A'\to A$ be 
an isomorphism from (B3) fixing every vertex of $A'\cap A$. 
Then $\rho_{A,c}(d)=\psi(d)$.
\item[$\bullet$] 
Let $A$ be an apartment of $X$ and let $S$ be a sector in~$A$. The \textit{sector retraction} 
$\rho_{A,S}:X\to A$ is defined as follows. Let $d$ be a chamber of $X$. Choose an apartment $A'$ containing $d$ and a subsector of $S$ (see \cite[VI, \S8, Theorem]{brown}), and let $\psi:A'\to A$ be an isomorphism from (B3) fixing every vertex of $A'\cap A$. Then $\rho_{A,S}(d)=\psi(d)$.
\end{enumerate}
Conceptually, $\rho_{A,c}$ ``radially flattens'' the building onto $A$ from the ``centre'' $c$
and $\rho_{A,S}$ flattens the building $X$ onto $A$ from a centre ``deep in the sector $S$''.
If $\delta(c,d)=w$ then $\rho_{A,c}(d)$ is the unique chamber of $A$ with $\delta(c,\rho_{A,c}(d))=w$.

If $\rho\colon X\to A$ is a retraction and $x$ is a vertex in $X$ then 
$\rho(x)$ is determined by $\rho(c)$ for a chamber $c$ which contains $x$. 
For a gallery $c_0,\ldots,c_{\ell}$ in $X$ let $\rho(c_0,\ldots,c_{\ell})$ be the pregallery $\rho(c_0),\ldots,\rho(c_{\ell})$ in~$A$. 

Viewing $\Sigma$ as an apartment of $X$ as in Section~\ref{subsection:buildings}, let
\begin{align}\label{eq:retractiondefn}
\begin{aligned}
\rho_{1}&:X\to\Sigma\quad\textrm{be the chamber retraction $\rho_{1}=\rho_{\Sigma,1}$, and}\\
\rho_{-\infty}&:X\to\Sigma\quad\textrm{be the sector retraction $\rho_{-\infty}=\rho_{\Sigma,w_0C_0}$},
\end{aligned}
\end{align}
where $C_0$ is the fundamental Weyl chamber of $\Sigma$, and 
$w_0$ is the longest element of $W_0$ so that $w_0C_0$ is the ``opposite'' Weyl chamber.

\begin{lemma}\label{lemma:retractions} Let $X$ be an affine building, and let $c$ and $d$ be chambers of $X$ with $\delta(c,d)=s_i$. Suppose that $\rho_{-\infty}(c)=v$. If $v\to vs_i$ is a positive crossing then 
\begin{align}\label{eq:case1}
\rho_{-\infty}\left(\beginpicture
\setcoordinatesystem units <0.8cm,0.8cm>         
\setplotarea x from -0.8 to 0.7, y from -0.5 to 0.5  
\put{$\scriptstyle{c}$}[br] at -0.6 0.1
\put{$\scriptstyle{d}$}[bl] at 0.6 0.1
\plot  0 -0.4  0 0.5 /
\arrow <5pt> [.2,.67] from -0.5 0 to 0.5 0   %
\endpicture\right)\quad=\quad\beginpicture
\setcoordinatesystem units <0.8cm,0.8cm>         
\setplotarea x from -0.8 to 0.7, y from -0.5 to 0.5  
\put{$\scriptstyle{-}$}[b] at -0.4 0.25
\put{$\scriptstyle{+}$}[b] at 0.4 0.25
\put{$\scriptstyle{v}$}[br] at -0.6 0.1
\put{$\scriptstyle{vs_i}$}[bl] at 0.6 0.1
\plot  0 -0.4  0 0.5 /
\arrow <5pt> [.2,.67] from -0.5 0 to 0.5 0   %
\endpicture\qquad\textrm{for all $d$ with $\delta(c,d)=s_i$}.
\end{align}
If $v\to vs_i$ is a negative crossing then there is a unique chamber $d_0$ of $X$ with $\delta(c,d_0)=s_i$ and $\rho_{-\infty}(d_0)=vs_i$, and
\begin{align}\label{eq:case2}
\rho_{-\infty}\left(\beginpicture
\setcoordinatesystem units <0.8cm,0.8cm>         
\setplotarea x from -0.8 to 0.7, y from -0.5 to 0.5  
\put{$\scriptstyle{c}$}[br] at -0.6 0.1
\put{$\scriptstyle{d}$}[bl] at 0.6 0.1
\plot  0 -0.4  0 0.5 /
\arrow <5pt> [.2,.67] from -0.5 0 to 0.5 0   %
\endpicture\right)\quad=\quad\begin{cases}
\beginpicture
\setcoordinatesystem units <0.8cm,0.8cm>         
\setplotarea x from -0.8 to 0.7, y from -0.5 to 0.5  
\put{$\scriptstyle{+}$}[b] at -0.4 0.25
\put{$\scriptstyle{-}$}[b] at 0.4 0.25
\put{$\scriptstyle{v}$}[br] at -0.6 0.1
\put{$\scriptstyle{vs_i}$}[bl] at 0.6 0.1
\plot  0 -0.4  0 0.5 /
\arrow <5pt> [.2,.67] from -0.5 0 to 0.5 0   %
\endpicture,&\textrm{if $d=d_0$,}\\
\beginpicture
\setcoordinatesystem units <0.8cm,0.8cm>         
\setplotarea x from -0.8 to 0.7, y from -0.5 to 0.5  
\put{$\scriptstyle{+}$}[b] at -0.4 0.35
\put{$\scriptstyle{-}$}[b] at 0.4 0.35
\put{$\scriptstyle{v}$}[br] at -0.6 0.1
\put{$\scriptstyle{vs_i}$}[bl] at 0.6 0.1
\plot  0 -0.4  0 0.6 /
\plot -0.5 0  -0.05 0 /
\arrow <5pt> [.2,.67] from -0.05 0.1 to -0.5 0.1   %
\plot -0.05 0 -0.05 0.1 /
\endpicture,&\textrm{if $d\neq d_0$.}
\end{cases}
\end{align}
\end{lemma}

\begin{proof} Suppose that $\delta(c,d)=s_i$ and $\rho_{-\infty}(c)=v$. Let $A$ be an apartment containing $c$ and a subsector $S$ of $w_0C_0$. Let $H$ be the wall of $A$ determined by the cotype $i$ panel of $c$. Thus $H$ divides $A$ into two ``half apartments'' $A^-$ and $A^+$, with $A^-\cap A^+=H$. Let $A^-$ be the half apartment which contains a subsector of~$S$. (See the diagram below).

Let $\psi:A\to\Sigma$ be an isomorphism from (B3) fixing $A\cap \Sigma$. The negative side (cf. Section~\ref{subsect:alcovewalk}) of the hyperplane $\psi(H)$ of $\Sigma$ is $\psi(A^-)$, because $\psi(A^-)$ contains a subsector of $w_0C_0$. If $v\to vs_i$ is a positive crossing then $v=\rho_{-\infty}(c)=\psi(c)$ is on the negative side of $\psi(H)$ (see (\ref{eq:symbols})), and therefore $c\in A^-$. If $v\to vs_i$ is a negative crossing then $c\in A^+$. These two cases are illustrated.
$$\beginpicture
    \setcoordinatesystem units <0.4cm,0.3cm>         
    \setplotarea x from -9 to 9, y from -4.5 to 4.5  
    \plot -9 -3 -3 3 9 3 3 -3 -9 -3 /
    \plot -1 -1 -3 -1 -2 -2 /
    \plot -1 -1 0 -2 -2 -2 /
    \plot -1 -1 -2.5 0 -2 -2 /
    \put{\small{$A^+$}} at -6 -4.5
    \put{\small{$A^-$}} at 0 -4.5
    \put{\small{$H$}} at -3.5 -4.5
    \put{\small{$S$}} at -0.8 2
    \put{\small{$c$}} at 0.5 -2
    \put{\small{$d$}} at -5 -0.5  
    \put{\small{$A$}} at 6 -2 
    \arrow <4pt> [.2,.67] from -4.5 -0.4 to -2.9 0   %
    \arrow <4pt> [.2,.67] from -4.5 -0.7 to -3.3 -1   %
    \setplotsymbol({\tiny{$\bullet$}})
    \plot -3 1 0 4 /
    \plot -3 1 8 1 /
    \plot -3.5 -3.5 4 4 /
    \hshade 1 1 7   3 3 9 /
    \put{(The case $c\in A^-$)} at -1 -6.5
    \endpicture\quad
    \beginpicture
    \setcoordinatesystem units <0.4cm,0.3cm>         
    \setplotarea x from -9 to 9, y from -4.5 to 4.5  
    \plot -9 -3 -3 3 9 3 3 -3 -9 -3 /
    \plot -1 -1 -3 -1 -2 -2 /
    \plot -1 -1 0 -2 -2 -2 /
    \plot -1 -1 -2.5 0 -2 -2 /
    \put{\small{$A^+$}} at -6 -4.5
    \put{\small{$A^-$}} at 0 -4.5
    \put{\small{$H$}} at -3.5 -4.5
    \put{\small{$S$}} at -0.8 2
    \put{\small{$d$}} at -5 0  
    \arrow <4pt> [.2,.67] from -4.5 0 to -2.9 0   %
    \put{\small{$d_0$}} at 0.7 -2 
    \put{\small{$c$}} at -3.5 -1.6 
    \put{\small{$A$}} at 6 -2 
    \setplotsymbol({\tiny{$\bullet$}})
    \plot -3 1 0 4 /
    \plot -3 1 8 1 /
    \plot -3.5 -3.5 4 4 /
    \hshade 1 1 7   3 3 9 /
    \put{(The case $c\in A^+$)} at -1 -6.5
    \endpicture
$$
If $c\in A^-$ then there is an apartment $A_1$ which contains $A^-\cup d$ (see \cite[Theorem~3.6]{ronan}). In particular $A_1$ contains a subsector of $S$, and thus a subsector of $w_0C_0$, and so $\rho_{-\infty}(d)=\psi_1(d)$ where $\psi_1:A_1\to \Sigma$ is any isomorphism fixing $A_1\cap \Sigma$. But $\psi_1(c)=\rho_{-\infty}(c)=v$, and so since $\psi_1(d)$ must be $i$-adjacent to (and distinct from) $\psi_1(c)$ we have $\rho_{-\infty}(d)=vs_i$, verifying~(\ref{eq:case1}).

Suppose that $c\in A^+$. If $d=d_0$ is the unique chamber of $A$ with $\delta(c,d_0)=s_i$ then $\rho_{-\infty}(d_0)=\psi(d)$ (where $\psi:A\to\Sigma$ is as above). But $\psi(d_0)$ must be $i$-adjacent to (and distinct from) $\psi(c)=\rho_{-\infty}(c)=v$, and so $\rho_{-\infty}(d_0)=vs_i$, verifying the first case of (\ref{eq:case2}). If $d\neq d_0$, then by analysis of the $c\in A^-$ case (with $d_0$ playing the role of $c$) we have $\rho_{-\infty}(d)=v=\rho_{-\infty}(c)$, verifying the second case of (\ref{eq:case2}).
\end{proof}

Suppose that $X$ is a regular affine building, and let $c$ be a chamber of $X$. 
Let $\mathbf{q}_i$ be an index set labelling the chambers $d$ such that 
$\delta(c,d)=s_i$ which appear in Lemma~\ref{lemma:retractions}. 
Assume that $\mathbf{q}_i$ has a distinguished element~$0$, 
which is the label for the chamber $d_0$ from Lemma~\ref{lemma:retractions}. Thus,
\begin{align}\label{eq:ret}
\textrm{if}\quad\rho_{-\infty}\bigg(\beginpicture
\setcoordinatesystem units <0.8cm,0.8cm>         
\setplotarea x from -0.8 to 0.7, y from -0.5 to 0.5  
\put{$\scriptstyle{i}$}[b] at 0 0.6
\put{$\scriptstyle{c}$}[br] at -0.6 0.1
\put{$\scriptstyle{d}$}[bl] at 0.6 0.1
\plot  0 -0.4  0 0.5 /
\arrow <5pt> [.2,.67] from -0.5 0 to 0.5 0   %
\endpicture\bigg)=\begin{cases}
\beginpicture
\setcoordinatesystem units <0.8cm,0.8cm>         
\setplotarea x from -0.8 to 0.7, y from -0.5 to 0.5  
\put{$\scriptstyle{-}$}[b] at -0.4 0.25
\put{$\scriptstyle{+}$}[b] at 0.4 0.25
\put{$\scriptstyle{v}$}[r] at -0.6 0.1
\put{$\scriptstyle{vs_i}$}[l] at 0.6 0.1
\plot  0 -0.4  0 0.5 /
\arrow <5pt> [.2,.67] from -0.5 0 to 0.5 0   %
\endpicture,&\quad\textrm{then}\quad z\in \mathbf{q}_i\\
\beginpicture
\setcoordinatesystem units <0.8cm,0.8cm>         
\setplotarea x from -0.8 to 0.7, y from -0.5 to 0.5  
\put{$\scriptstyle{+}$}[b] at -0.4 0.25
\put{$\scriptstyle{-}$}[b] at 0.4 0.25
\put{$\scriptstyle{v}$}[r] at -0.6 0.1
\put{$\scriptstyle{vs_i}$}[l] at 0.6 0.1
\plot  0 -0.4  0 0.5 /
\arrow <5pt> [.2,.67] from -0.5 0 to 0.5 0   %
\endpicture,&\quad\textrm{then}\quad z=0\\
\beginpicture
\setcoordinatesystem units <0.8cm,0.8cm>         
\setplotarea x from -0.8 to 0.7, y from -0.5 to 0.5  
\put{$\scriptstyle{+}$}[b] at -0.4 0.35
\put{$\scriptstyle{-}$}[b] at 0.4 0.35
\put{$\scriptstyle{v}$}[r] at -0.6 0.1
\put{$\scriptstyle{vs_i}$}[l] at 0.6 0.1
\plot  0 -0.4  0 0.6 /
\plot -0.5 0  -0.05 0 /
\arrow <5pt> [.2,.67] from -0.05 0.1 to -0.5 0.1   %
\plot -0.05 0 -0.05 0.1 /
\endpicture,&\quad\textrm{then}\quad z\in \mathbf{q}_i\backslash\{0\},
\end{cases}
\end{align}
where $v = \rho_{-\infty}(c)$ and $z$ is the label of $d$.
The same set $\mathbf{q}_i$ can be used for every chamber $c$ (by regularity). Note that if $|\mathbf{q}_i|=1$ (that is, the cotype $i$ panels of $X$ are ``thin'') then $\mathbf{q}_i=\{0\}$, and so there will be no $i$-folds.

\begin{prop}\label{prop:b} Let $X$ be a regular affine building, and let the sets $\mathbf{q}_i$ be as defined above. Let $\vec w=s_{i_1}\cdots s_{i_{\ell}}$ be a reduced expression, and let $v\in W$. Then the set
$$
\cG(\vec w)_v=\left\{\begin{matrix}\textrm{galleries $c_0,\ldots,c_{\ell}$ in $X$ of type $\vec w$}\\
\textrm{with initial chamber $c_0=1\in\Sigma$}\end{matrix}\,\,\bigg|\,\,\rho_{-\infty}(c_{\ell})=v\right\}
$$
is in bijection with the set of all $\mathbf{q}$-labelled positively folded alcove walks of type~$\vec w$ with end alcove~$v$, where a $\mathbf{q}$-labelled alcove walk of type~$\vec w$ (starting at $1$) is made up of the symbols
\begin{align*}
\beginpicture
\setcoordinatesystem units <0.8cm,0.8cm>         
\setplotarea x from -0.8 to 0.7, y from -0.5 to 0.5  
\put{$\scriptstyle{-}$}[b] at -0.4 0.25
\put{$\scriptstyle{+}$}[b] at 0.4 0.25
\put{$\scriptstyle{z}$}[tl] at 0.2 -0.3
\put{$\scriptstyle{x}$}[br] at -0.6 0.1
\put{$\scriptstyle{xs_i}$}[bl] at 0.6 0.1
\plot  0 -0.4  0 0.5 /
\arrow <5pt> [.2,.67] from -0.5 0 to 0.5 0   %
\put{with $z\in \mathbf{q}_i$} at 0 -1
\endpicture
\qquad\qquad
\beginpicture
\setcoordinatesystem units <0.8cm,0.8cm>         
\setplotarea x from -0.8 to 0.7, y from -0.5 to 0.5  
\put{$\scriptstyle{-}$}[b] at -0.4 0.35
\put{$\scriptstyle{+}$}[b] at 0.4 0.35
\put{$\scriptstyle{z}$}[tl] at 0.2 -0.3
\put{$\scriptstyle{x}$}[bl] at 0.6 0.1
\put{$\scriptstyle{xs_i}$}[br] at -0.6 0.1
\plot  0 -0.4  0 0.6 /
\plot 0.5 0  0.05 0 /
\arrow <5pt> [.2,.67] from 0.05 0.1 to 0.5 0.1   %
\plot 0.05 0 0.05 0.1 /
\put{with $z\in \mathbf{q}_i\backslash\{0\}$} at 0 -1
\endpicture
\qquad\quad\textrm{and}\qquad\qquad
\beginpicture
\setcoordinatesystem units <0.8cm,0.8cm>         
\setplotarea x from -0.8 to 0.7, y from -0.5 to 0.5  
\put{$\scriptstyle{-}$}[b] at -0.4 0.25
\put{$\scriptstyle{+}$}[b] at 0.4 0.25
\put{$\scriptstyle{z}$}[tl] at 0.2 -0.3
\put{$\scriptstyle{x}$}[bl] at 0.6 0.1
\put{$\scriptstyle{xs_i}$}[br] at -0.6 0.1
\plot  0 -0.4  0 0.5 /
\arrow <5pt> [.2,.67] from 0.5 0 to -0.5 0   %
\put{with $z=0$} at 0 -1 
\endpicture
\end{align*}
where the $k$th step has $i=i_k$ for $k=1,\ldots,\ell$. (If $|\mathbf{q}_i|=1$ then there are no $i$-folds).
\end{prop}

\begin{proof} By Lemma~\ref{lemma:retractions}, if $X$ is thick then $\rho_{-\infty}(\cG(\vec w)_v)=\cP(\vec w)_v$ (the set of positively folded alcove walks of type $\vec w$ with end alcove~$v$). By labelling the paths of $\cP(\vec w)_v$ and using (\ref{eq:ret}) we obtain the bijection of Proposition~\ref{prop:b} (the labelling also corrects for the possibility that the building has thin panels).
\end{proof}

\begin{thm}\label{thm:retractions} Let $X$ be a regular affine building of irreducible type and let $\mathbf{q}_i$ be as defined just above (\ref{eq:ret}). For $\la\in P^+$ and $\mu\in P$ define
$$
\cP_\mathbf{q}(\vec\la)_{\mu}=\left\{\begin{matrix}\textrm{$\mathbf{q}$-labelled positively folded alcove walks}\\
\textrm{of type $\vec u\cdot\vec m_{\la}$ with end alcove in $t_{\mu}W_0$}\end{matrix} \,\,\bigg|\,\, u\in W_0^{\la}\right\},
$$
where, as in the statement of Theorem~$\mathbf{C}$, $\vec u\cdot \vec m_\lambda$ is a minimal length walk to
$ut_\lambda W_0$.  Then there is a bijection
$$
\cP_{\mathbf{q}}(\vec\la)_{\mu}\longleftrightarrow \{x\in X\mid \rho_{1}(x)\in W_0\la\textrm{ and }\rho_{-\infty}(x)=\mu\},
$$
where the retractions $\rho_1,\rho_{-\infty}:X\to\Sigma$ are as in (\ref{eq:retractiondefn}).
\end{thm}

\begin{proof} Recall that $\tilde{W}=W\rtimes \Omega$, and that an alcove walk of type $s_{i_1}\cdots s_{i_{\ell}}\gamma$ is really an alcove walk in $\fh_{\RR}\times\Omega$. It is convenient to define a similar notion for buildings: Let $\tilde{X}=X\times\Omega$. A gallery of type $s_{i_1}\cdots s_{i_{\ell}}\gamma$ in $\tilde{X}$ with initial chamber $(c_0,\gamma_0)$ is a sequence
$
(c_0,\gamma_0),\ldots,(c_{\ell},\gamma_0),(c_{\ell},\gamma_0\gamma)$ with $c_0,\ldots,c_{\ell}$ a gallery of type $s_{i_1}\cdots s_{i_{\ell}}$ in $X$. The retraction $\rho_{-\infty}:X\to\Sigma$ gives a retraction $\tilde{\rho}_{-\infty}:\tilde{X}\to\Sigma\times\Omega$ by $\tilde{\rho}_{-\infty}(c,\gamma)=(\rho_{-\infty}(c),\gamma)$.

If $X$ is thick, then by Lemma~\ref{lemma:retractions} we have $\tilde{\rho}_{-\infty}(\cG(\vec\la)_{\mu})=\cP(\vec\la)_{\mu}$, where $\cP(\vec\la)_{\mu}$ is as in Theorem~$\mathbf{C}$ and
$$
\cG(\vec\la)_{\mu}=\left\{\begin{matrix}\textrm{galleries $\tilde{c}_0,\ldots,\tilde{c}_{\ell}$ in $\tilde{X}$ of type $\vec u\cdot\vec m_{\la}$}\\
\textrm{with initial chamber $\tilde{c}_0=(1,1)\in\Sigma\times \Omega$}\end{matrix}\,\,\bigg|\,\,u\in W_0^{\la}\textrm{ and }\tilde{\rho}_{-\infty}(\tilde{c}_{\ell})\in t_{\mu}W_0\right\}.
$$ 
By labelling the paths in $\cP(\vec\la)_{\mu}$ we obtain a bijection $\cG(\vec\la)_{\mu}\leftrightarrow\cP_{\mathbf{q}}(\vec\la)_{\mu}$, and as in Proposition~\ref{prop:b} this labelling corrects for the possibility of thin panels. Mapping galleries in $\tilde{X}$ to their end vertices and then projecting $\tilde{X}\to X$ gives a bijection 
$$
\cG(\vec\la)_{\mu}\leftrightarrow\{x\in X\mid \rho_1(x)\in W_0\la\textrm{ and }\rho_{-\infty}(x)=\mu\},
$$ 
completing the proof.
\end{proof}


\section{The path formula for $P_{\la}(x,q^{-1})$}\label{app:pathformula}

Since our set $\cP(\vec\la)_{\mu}$ is slightly different from the set used in \cite{ram2} and \cite{schwer}, we provide a condensed sketch of the derivation of formula (\ref{eq:macpath}), following~\cite{ram} and~\cite{ram2}. The \textit{affine Hecke algebra} $\cH$ is the algebra over $\ZZ[q^{\frac{1}{2}},q^{-\frac{1}{2}}]$ with generators $T_w$, $w\in\tilde{W}$, and relations
\begin{align*}
T_{u}T_v&=T_{uv}&&\textrm{if $\ell(uv)=\ell(u)+\ell(v)$, and}\\
T_wT_{s_j}&=T_{ws_j}+(q^{\frac{1}{2}}-q^{-\frac{1}{2}})T_w&&\textrm{if $\ell(ws_j)<\ell(w)$.}
\end{align*}
For any choice of expression $\vec
v=s_{i_1}\cdots s_{i_\ell}\gamma$ (not necessarily reduced, with $\gamma\in\Omega$), let
\begin{align*}
x_{v}=T_{s_{i_1}}^{\epsilon_1}\cdots
T_{s_{i_\ell}}^{\epsilon_\ell}T_{\gamma}\qquad\textrm{where}\qquad\epsilon_k
=\begin{cases} +1, &\textrm{if the $k$th step of $\vec v$ is
$\beginpicture
    \setcoordinatesystem units <0.8cm,0.8cm>         
    \setplotarea x from -0.8 to 0.8, y from -0.5 to 0.5  
    \put{$\scriptstyle{-}$}[b] at -0.4 0.25
    \put{$\scriptstyle{+}$}[b] at 0.4 0.25
    \plot  0 -0.4  0 0.5 /
    \arrow <5pt> [.2,.67] from -0.5 0 to 0.5 0   %
\endpicture$,}\\
-1, &\textrm{if the $k$th step of $\vec v$ is
$\beginpicture
    \setcoordinatesystem units <0.8cm,0.8cm>         
    \setplotarea x from -0.8 to 0.7, y from -0.5 to 0.5  
    \put{$\scriptstyle{-}$}[b] at -0.4 0.25
    \put{$\scriptstyle{+}$}[b] at 0.4 0.25
    \plot  0 -0.4  0 0.5 /
    \arrow <5pt> [.2,.67] from 0.5 0 to -0.5 0   %
\endpicture$\,.}
\end{cases}
\end{align*}
This does not depend on the expression for $v$ (see \cite[Theorem~3.1.1]{ulrich}) 
and $\{x_v\mid v\in\tilde{W}\}$ gives another basis of $\cH$. 
If $\vec w = s_{i_1}\cdots s_{i_\ell}\gamma$ is a reduced expression for $w$
then $T_w = T_{s_{i_1}}\cdots T_{s_{i_\ell}}T_{\gamma}$.
The relation $T_{s_i} = T_{s_i}^{-1}+(q^{\frac12} - q^{-\frac12})$ and an induction on $\ell(w)$ give that
\begin{align}\label{eq:walk}
T_{w}=\sum_{p\in P(\vec w)}(q^{\frac{1}{2}}-q^{-\frac{1}{2}})^{f(p)}x_{\mathrm{end}(p)}
=\sum_{p\in \cP(\vec w)} q^{\frac{f(p)}{2}}(1-q^{-1})^{f(p)}x_{\mathrm{end}(p)},
\end{align}
where the sum is over all positively folded alcove walks of type~$\vec w = s_{i_1}\cdots s_{i_\ell}\gamma$
and $\mathrm{end}(p)$ is the alcove where $p$ ends.

For $\mu\in P$ let $x^{\mu}=x_{t_{\mu}}$. Then $x^{\la}x^{\mu}=x^{\la+\mu}=x^{\mu}x^{\la}$ for all $\la,\mu\in P$. If
$$\mathbf{1}_0=\sum_{w\in W_0}q^{\frac{\ell(w)}{2}}T_w,\qquad\textrm{then $T_w\mathbf{1}_0=\mathbf{1}_0T_w=q^{\frac{\ell(w)}{2}}\mathbf{1}_0$ for all $w\in W_0$.}
$$
 The \textit{spherical Hecke algebra} is $\mathbf{1}_0\cH\mathbf{1}_0$. (Note that the spherical Hecke algebra is \textit{not} the Hecke algebra of the spherical Weyl group $W_0$). Since each $w\in\tilde{W}$ can be written as $w=ut_{\la}v$ with $\la\in P^+$ and $u,v\in W_0$ it follows that $\{\mathbf{1}_0x^{\la}\mathbf{1}_0\mid\la\in P^+\}$ is a basis of~$\mathbf{1}_0\cH\mathbf{1}_0$.

\begin{proof}[Proof of formula (\ref{eq:macpath})] After appropriate conversions in notations to match our definitions,
\cite[Theorem~2.9(a)]{ram} says that,
for $\lambda\in P^+$, the definition of $P_\lambda(x,q^{-1})$ in (\ref{eq:Pladefn}) is equivalent
to 
$$
P_{\la}(x,q^{-1})\mathbf{1}_0=
\frac{q^{-\ell(w_0)}}{W_{0\lambda}(q^{-1})}
\mathbf{1}_0 x^\lambda 
\mathbf{1}_0.
$$
Thus, by the definition of $\mathbf{1}_0$ and the fact that each $w\in W_0$ has a unique expression as $w=uv$ with $u\in W_0^{\la}$ and $v\in W_{0\la}$, and moreover $\ell(uv)=\ell(u)+\ell(v)$, we have
\begin{align*}
P_\lambda(x,q^{-1})\mathbf{1}_0
&= \frac{q^{-\ell(w_0)}} {W_{0\lambda}(q^{-1})} 
\sum_{u\in W_0^\lambda}\sum_{v\in W_{0\lambda}^{\vphantom{\la}}} q^{\frac12(\ell(u)+\ell(v))} T_uT_v
x^\lambda \mathbf{1}_0 \\
&= \frac{q^{-\ell(w_0)}} {W_{0\lambda}(q^{-1})} W_{0\lambda}(q) 
\sum_{u\in W_0^\lambda} q^{\frac12\ell(u)} T_u
x^\lambda \mathbf{1}_0 &&\textrm{(since $T_vx^{\la}=x^{\la}T_v$)}.
\end{align*}
Let $w_{0\la}$ be the longest element of $W_{0\la}$. Then $W_{0\la}(q)=q^{\ell(w_{0\la})}W_{0\la}(q^{-1})$, and so
\begin{align*}
P_{\la}(x,q^{-1})\mathbf{1}_0&=q^{\ell(w_{0\la})-\ell(w_0)}\sum_{u\in W_0^{\la}}q^{\frac{1}{2}\ell(u)}T_ux^{\la}\mathbf{1}_0=q^{\ell(m_{\la})-\ell(t_{\la})}\sum_{u\in W_0^{\la}}q^{\frac{1}{2}\ell(u)}T_ux^{\la}\mathbf{1}_0,
\end{align*}
where we have used $t_{\la}=m_{\la}w_0w_{0\la}$ to get $\ell(w_0)-\ell(w_{0\la})=\ell(t_{\la})-\ell(m_{\la})$. Since 
$$
x^{\la}\mathbf{1}_0=T_{t_{\la}}\mathbf{1}_0=T_{m_{\la}}T_{w_0w_{0\la}}\mathbf{1}_0=q^{\frac{1}{2}(\ell(w_0)-\ell(w_{0\la}))}T_{m_{\la}}\mathbf{1}_0=q^{\frac{1}{2}(\ell(t_{\la})-\ell(m_{\la}))}T_{m_{\la}}\mathbf{1}_0,
$$
using (\ref{eq:walk}) we have
\begin{align*}
P_{\la}(x,q^{-1})\mathbf{1}_0&= \sum_{u\in W_0^\lambda} 
q^{\frac{1}{2}(\ell(u)+\ell(m_{\la})-\ell(t_{\la}))} T_{um_{\la}}
\mathbf{1}_0 \\
&= \sum_{u\in W_0^\lambda} 
\sum_{p\in \cP(\vec u\cdot\vec m_\lambda)} 
q^{\frac{1}{2}(\ell(p)-\ell(t_{\la}))}q^{\frac{1}{2}f(p)}(1-q^{-1})^{f(p)} x_{\mathrm{end}(p)}\mathbf{1}_0.
\end{align*}
By (\ref{eq:signedlength}), if $\mathrm{end}(p)=t_{\mu}\varphi(p)$ then 
$
x_{\mathrm{end}(p)}\mathbf{1}_0=x^{\mu}T_{\varphi(p)^{-1}}^{-1}\mathbf{1}_0=q^{-\frac{1}{2}\ell(\varphi(p))}x^{\mu}\mathbf{1}_0,
$
and so
$$
P_{\la}(x,q^{-1})\mathbf{1}_0=\sum_{\mu\in P}\bigg(\sum_{p\in\cP(\vec \la)_{\mu}}q^{\frac{1}{2}(\ell(p)-\ell(t_{\la})+f(p)-\ell(\varphi(p)))}(1-q^{-1})^{f(p)}\bigg)x^{\mu}\mathbf{1}_0.
$$
By (\ref{eq:epsilformula}), (\ref{eq:dimension}), and (\ref{eq:signedlength}) we have
$
\ell(p)-\ell(t_{\la})+f(p)-\ell(\varphi(p))=2(\dim(p)-\langle\la+\mu,\rho\rangle).
$
\end{proof}

\begin{remark}\label{rem:diffpaths} For example, in type $A_2$ with $\la=\omega_1+\omega_2$ the expansion of $P_{\la}(x,q^{-1})$ is given by (\ref{eq:macpath}) and the 25 paths in Example~\ref{ex:paths}. We note that while our set $\cP(\vec\la)_{\mu}$ is very natural (particularly with respect to retractions in buildings; see Theorem~\ref{thm:retractions}), it does not give the most efficient expansion of $P_{\la}(x,q^{-1})$. Instead, let
$$
\cP'(\vec\la)_{\mu}=\left\{\begin{matrix}\textrm{positively folded alcove walks of type $\vec m_{\la}$ starting}\\
\textrm{at the alcove $u$ and with end alcove in $t_{\mu}W_0$}\end{matrix}\,\,\bigg|\,\, u\in W_0^{\la}\right\}.
$$
For example, in type $A_2$ with $\la=\omega_1+\omega_2$ there are $9$ paths in $\cP'(\vec\la)=\bigcup_{\mu\in P}\cP'(\vec\la)_{\mu}$:
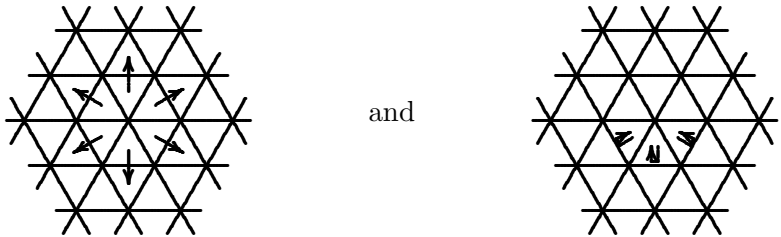
$$\beginpicture
\setcoordinatesystem units <0.3cm,0.3cm>         
\setplotarea x from -5.5 to 5.5, y from -5 to 5    
    \plot 1.732 -5  5.196 1 /
    \plot -0.577 -5  4.041 3 /
    \plot -2.887 -5  2.887 5 /
    \plot -4.041 -3  0.577 5 /
    \plot -5.196 -1  -1.732 5 /
    \plot -1.732 -5  -5.196 1 /
    \plot 0.577 -5  -4.041 3 /
    \plot 2.887 -5  -2.887 5 /
    \plot 4.041 -3  -0.577 5 /
    \plot 5.196 -1  1.732 5 /
    \plot -3.2 4  3.2 4 /
    \plot -4.4 2  4.4 2 /
    \plot -5.4 0  5.4 0 /
    \plot -3.2 -4  3.2 -4 /
    \plot -4.4 -2  4.4 -2 /
    \arrow <5pt> [.2,.67] from 0 1.3 to 0 2.8
    \arrow <5pt> [.2,.67] from 0 -1.3 to 0 -2.8
    \arrow <5pt> [.2,.67] from 1.2 0.7 to 2.4 1.4
    \arrow <5pt> [.2,.67] from 1.2 -0.7 to 2.4 -1.4
    \arrow <5pt> [.2,.67] from -1.2 0.7 to -2.4 1.4
    \arrow <5pt> [.2,.67] from -1.2 -0.7 to -2.4 -1.4
\endpicture\qquad\qquad\textrm{and}\qquad\qquad\beginpicture
\setcoordinatesystem units <0.3cm,0.3cm>         
\setplotarea x from -5.5 to 5.5, y from -5 to 5    
    \plot 1.732 -5  5.196 1 /
    \plot -0.577 -5  4.041 3 /
    \plot -2.887 -5  2.887 5 /
    \plot -4.041 -3  0.577 5 /
    \plot -5.196 -1  -1.732 5 /
    \plot -1.732 -5  -5.196 1 /
    \plot 0.577 -5  -4.041 3 /
    \plot 2.887 -5  -2.887 5 /
    \plot 4.041 -3  -0.577 5 /
    \plot 5.196 -1  1.732 5 /
    \plot -3.2 4  3.2 4 /
    \plot -4.4 2  4.4 2 /
    \plot -5.4 0  5.4 0 /
    \plot -3.2 -4  3.2 -4 /
    \plot -4.4 -2  4.4 -2 /
    \plot 0.15 -1.15 0.15 -1.8 /
    \plot 0.15 -1.8 -0.15 -1.8 /
    \arrow <5pt> [.2,.67] from -.15 -1.8 to -.15 -1.15
    \put{
\beginpicture
\setcoordinatesystem units <0.6cm,0.6cm>
\setplotarea x from -1 to 1, y from -1 to 1  
    \arrow <5pt> [.2,.67] from 0.3 -0.15 to 0 0
    \plot 0.3 -0.15 0.225 -0.3 /
    \plot 0.225 -0.3 -0.05 -0.13 /
\endpicture} at 0.9 -0.5
    \put{\beginpicture 
\setcoordinatesystem units <0.6cm,0.6cm>
\setplotarea x from -1 to 1, y from -1 to 1  
    \plot 0.05 -0.13 -0.225 -0.3 /
    \plot -0.225 -0.3 -0.3 -0.15 /
    \arrow <5pt> [.2,.67] from -0.3 -0.15 to 0 0
\endpicture} at -1.1 -0.5
\endpicture$$
We have
\begin{align}\label{eq:macpath2}
P_{\la}(x,q^{-1})=q^{\frac{1}{2}(\ell(t_{\la})-\ell(m_{\la}))}\sum_{\mu\in P}\bigg(\sum_{p\in \cP'(\vec\la)_{\mu}}q^{-\frac{1}{2}(\ell(\iota(p))+\ell(\varphi(p)))}(q^{\frac{1}{2}}-q^{-\frac{1}{2}})^{f(p)}\bigg)x^{\mu},
\end{align}
where $\iota(p)\in W_0^{\la}$ is the \textit{initial alcove} of $p$. In type $A_2$ with $\la=\omega_1+\omega_2$ formula (\ref{eq:macpath2}) quickly gives
\begin{align*}
P_{\la}(x,q^{-1})=x^{\la}+x^{s_1\la}+x^{s_2\la}+x^{s_1s_2\la}+x^{s_2s_1\la}+x^{w_0\la}+(2+q^{-1})(1-q^{-1}).
\end{align*}

The proof of (\ref{eq:macpath2}) is exactly analogous to the proof of (\ref{eq:macpath}) given above, except the initial expansion of $\mathbf{1}_0$ is replaced by the expansion
$$
\mathbf{1}_0=q^{\ell(w_0)}\sum_{w\in W_0}q^{-\frac{1}{2}\ell(w)}T_{w^{-1}}^{-1}=q^{\ell(w_0)}\sum_{w\in W_0}q^{-\frac{1}{2}\ell(w)}x_v,
$$
which is proved by noticing that
$$
\mathbf{1}_0=\sum_{w\in W_0}q^{\frac{1}{2}\ell(w_0w)}T_{w_0w}=q^{\frac{1}{2}\ell(w_0)}T_{w_0}\sum_{w\in W_0}q^{-\frac{1}{2}\ell(w)}T_{w^{-1}}^{-1}
$$
and then multiplying by $q^{\frac{1}{2}\ell(w_0)}T_{w_0}^{-1}$.
\end{remark}

\end{appendix}

\bibliography{Bib.bib}

\begin{thebibliography}{10}

\bibitem{andersonpoly}
J.~E. {A}nderson.
\newblock A polytope calculus for semisimple groups.
\newblock {\em Duke Math. J.}, 116(3):567--588, 2003.
\newblock arXiv:math/0110225.

\bibitem{bourbaki}
N.~{B}ourbaki.
\newblock {\em {L}ie Groups and {L}ie Algebras, Chapters 4--6}.
\newblock Elements of Mathematics. Springer-Verlag, Berlin Heidelberg New York,
  2002.

\bibitem{brown}
K.~{B}rown.
\newblock {\em {B}uildings}.
\newblock Springer-Verlag, New York, 1989.

\bibitem{litt}
S.~{G}aussent and P.~{L}ittelmann.
\newblock {LS}-galleries, the path model and {MV}-cycles.
\newblock {\em Duke Math. J.}, 127(1):35--88, 2005.
\newblock arXiv:math/0307122v3.

\bibitem{ulrich}
U.~{G}\"{o}rtz.
\newblock {A}lcove walks and nearby cycles on affine flag manifolds.
\newblock {\em J. Alg. Comb.}, 26:415--430, 2007.
\newblock arXiv:math/0610839.

\bibitem{haines}
T.~{H}aines.
\newblock {O}n matrix coefficients of the {S}atake isomorphism: {C}omplements
  to the paper of {M}. {R}apoport.
\newblock {\em Manuscripta Math.}, 101:167--174, 2000.

\bibitem{petra}
P.~{H}itzelberger.
\newblock {K}ostant convexity for affine buildings.
\newblock {\em {P}reprint}, 2008.
\newblock arXiv:math/0701094v2.

\bibitem{h2}
J.~E. {H}umphreys.
\newblock {\em {I}ntroduction to {L}ie Algebras and Representation Theory},
  volume~9 of {\em Graduate Texts in Mathematics}.
\newblock Springer-Verlag, New York-Berlin, 1978.

\bibitem{kamnitzer2}
J.~{K}amnitzer.
\newblock {M}irkovi\'{c}-{V}ilonen cycles and polytopes.
\newblock {\em {P}reprint}, 2005.
\newblock arXiv:math/0501365v2.

\bibitem{kamnitzer}
J.~{K}amnitzer.
\newblock The crystal structure on the set of {M}irkovi\'{c}-{V}ilonen
  polytopes.
\newblock {\em Adv. Math.}, 215:66--93, 2007.
\newblock arXiv:math/0505398v2.

\bibitem{littelmanncrystals}
P.~{L}ittelmann.
\newblock {C}ones, crystals, and patterns.
\newblock {\em Transformation groups}, 3(2):145--179, 1998.

\bibitem{m}
I.~G. {M}acdonald.
\newblock {\em {A}ffine {H}ecke Algebras and Orthogonal Polynomials}, volume
  157 of {\em Cambridge Tracts in Mathematics}.
\newblock Cambridge Univ. Press, Cambridge, 2003.

\bibitem{MV3}
I.~{M}irkovi\'{c} and K.~{V}ilonen.
\newblock Geometric {L}anglands duality and representations of algebraic groups
  over commutative rings.
\newblock {\em Ann. of Math.}, 166(1):95--143, 2007.
\newblock arXiv:math/0401222v4.

\bibitem{ram}
K.~{N}elsen and A.~{R}am.
\newblock {K}ostka-{F}oulkes polynomials and {M}acdonald spherical functions.
\newblock {\em Surveys in Combinatorics, 2003 (Bangor), London Math. Soc.
  Lecture Note Ser.}, 307:325--370, 2003.
\newblock arXiv:math/0401298.

\bibitem{p2}
J.~{P}arkinson.
\newblock {S}pherical harmonic analysis on affine buildings.
\newblock {\em Math. Z.}, 253(3):571--606, 2006.
\newblock arXiv:math/0604058.

\bibitem{ramparkinsonschwer}
J.~{P}arkinson, A.~{R}am, and C.~{S}chwer.
\newblock {C}ombinatorics in affine flag varieties.
\newblock {\em {T}o appear in {J}. {A}lgebra}, 2008.
\newblock arXiv:0801.0709.

\bibitem{ram2}
A.~{R}am.
\newblock Alcove walks, {H}ecke algebras, spherical functions, crystals and
  column strict tableaux.
\newblock {\em Pure Appl. Math. Quart.}, 2(4):963--1013, 2006.
\newblock arXiv:math/0601343.

\bibitem{rapoport}
M.~{R}apoport.
\newblock {A} positivity property of the {S}atake isomorphism.
\newblock {\em Manuscripta Math.}, 101:153--166, 2000.

\bibitem{ronan}
M.~{R}onan.
\newblock {\em {L}ectures on Buildings}.
\newblock Perspectives in Mathematics. Academic Press, 1989.

\bibitem{schwer}
C.~{S}chwer.
\newblock {G}alleries, {H}all-{L}ittlewood polynomials and structure constants
  of the spherical {H}ecke algebra.
\newblock {\em International Mathematics Research Notices}, 2006:1--31, 2006.
\newblock arXiv:math/0506287v3.

\bibitem{steinberg}
R.~{S}teinberg.
\newblock {\em Lectures on {C}hevalley groups}.
\newblock Yale University lectures. 1967.

\bibitem{tupan}
A.~{T}upan.
\newblock {P}ositivity for some {S}atake coefficients.
\newblock {\em Manuscripta Math.}, 112:191--195, 2003.

\end{thebibliography}
\bibliographystyle{plain}

\end{document}